\documentclass[a4paper,11pt,leqno]{article}

\usepackage[latin1]{inputenc}
\usepackage[T1]{fontenc} 
\usepackage[english]{babel}
\usepackage{verbatim}
\usepackage{calligra}
\usepackage{enumerate}
\usepackage{dsfont}
\usepackage{amsfonts}
\usepackage{amsmath}
\usepackage{amsthm}
\usepackage{amssymb}
\usepackage{mathtools}
\usepackage{mathrsfs}
\usepackage{color}
\usepackage{graphicx}
\usepackage{bm}
\usepackage{srcltx}

\usepackage{tikz}
\usepackage[retainorgcmds]{IEEEtrantools}

\usepackage{graphicx}		
\usepackage[hypcap=false]{caption}

\theoremstyle{definition}
\newtheorem{defin}{Definition}[section]

\newtheorem{rmk}[defin]{Remark}

\theoremstyle{plane}
\newtheorem{thm}[defin]{Theorem}
\newtheorem{prop}[defin]{Proposition}
\newtheorem{cor}[defin]{Corollary}
\newtheorem{lemma}[defin]{Lemma}

\newcommand{\tbf}{\textbf}

\newcommand{\tsl}{\textsl}

\newcommand{\mc}{\mathcal}
\newcommand{\mf}{\mathfrak}

\newcommand{\veps}{\varepsilon}
\newcommand{\what}{\widehat}

\newcommand{\vphi}{\varphi}
\newcommand{\oline}{\overline}

\newcommand{\s}{\sigma}
\renewcommand{\t}{\tau}

\renewcommand{\o}{\omega}

\newcommand{\lan}{\langle}
\newcommand{\ran}{\rangle}

\newcommand{\R}{\mathbb{R}}

\newcommand{\N}{\mathbb{N}}
\newcommand{\Z}{\mathbb{Z}}
\newcommand{\T}{\mathbb{T}}

\renewcommand{\div}{{\rm div}\,}

\newcommand{\curl}{{\rm curl}\,}

\newcommand{\Id}{{\rm Id}\,}
\newcommand{\supp}{{\rm supp}\,}

\newcommand{\dt}{ \, {\rm d} t}

\newcommand{\weak}{{\rm w}}

\allowdisplaybreaks

\def\d{\partial}
\def\div{{\rm div}\,}
\newcommand{\dd}{{\rm d}}

\begin{document}

\title{\textsc{\Large{\textbf{Geometric blow-up criteria for the non-homogeneous incompressible Euler equations in 2-D}}}}

\author{\normalsize \textsl{Francesco Fanelli}$\,^{1,2,3}$ 
\vspace{.5cm} \\
\footnotesize{$\,^{1}\;$ \textsc{BCAM -- Basque Center for Applied Mathematics}} \\ 
{\footnotesize Alameda de Mazarredo 14, E-48009 Bilbao, Basque Country, SPAIN} \vspace{.2cm} \\
\footnotesize{$\,^{2}\;$ \textsc{Ikerbasque -- Basque Foundation for Science}} \\  
{\footnotesize Plaza Euskadi 5, E-48009 Bilbao, Basque Country, SPAIN} \vspace{.2cm} \\
\footnotesize{$\,^{3}\;$ \textsc{Universit\'e Claude Bernard Lyon 1}, {\it Institut Camille Jordan -- UMR 5208}} \\ 
{\footnotesize 43 blvd. du 11 novembre 1918, F-69622 Villeurbanne cedex, FRANCE} \vspace{.3cm} \\
%
%
\footnotesize{Email address: \ttfamily{ffanelli@bcamath.org}}
\vspace{.2cm}
}

\date\today

\maketitle

\subsubsection*{Abstract}
{\footnotesize 

This paper concerns the study of the incompressible Euler equations with variable density, in the case of space dimension $d=2$.
Contrarily to their homogeneous (constant density) counterpart, those equations are not known to be well-posed globally in time.
A classical blow-up/continuation criterion for smooth solutions relies on the control of the Lipschitz norm
of the velocity field $u$.

Here we show that, for establishing blow-up or continuation of solutions, it is enough to determine a control of
$\nabla u$ only along the direction $X=\nabla^\perp\rho$, where $\rho$ represents the density of the fluid.
Our results deal with both the subcritical regularity and critical regularity frameworks. They rely on a novel approach to study regularity of solutions
for the density-dependent incompressible Euler equations.
Besides, they allow to recover the global well-posedness for $\rho\equiv {\rm cst}$ as a particular case.

}

\paragraph*{\small 2020 Mathematics Subject Classification:}{\footnotesize 35Q31 
(primary);
35B60, 
76B70 
(secondary).}

\paragraph*{\small Keywords: }{\footnotesize incompressible Euler equations; density variations; geometric regularity; blow-up and continuation criteria.
}

%
%

\section{Introduction} \label{s:intro}

In this paper, we consider the  \emph{non-homogeneous incompressible Euler equations}, namely the system
\begin{equation} \label{eq:dd-E}
\left\{\begin{array}{l}
\d_t\rho\,+\,u\cdot\nabla\rho\,=\,0 \\[1ex]
\rho\,\d_tu\,+\,\rho\,u\cdot\nabla u\,+\,\nabla\Pi\,=\,0 \\[1ex]
\div u\,=\,0\,.
       \end{array}
\right.
\end{equation}
This system can be used to describe the motion of an incompressible inviscid fluid presenting density variations. The first equation is derived from the principle of
conservation of mass; the second equation can be derived \tsl{e.g.} from the Newton law applied to an infinitesimal fluid parcel
and represents the conservation of linear momentum; finally, the last equation encodes the incompressibility assumption on the flow.

We define system \eqref{eq:dd-E} for
\[
(t,x)\,\in\,\R_+\times\Omega\,,\qquad\qquad \mbox{ with }\quad \Omega\,=\,\R^d\;\mbox{ or }\;\T^d\,,
\]
where $t$ is the time variable and $x$ the space variable. For the sake of the present discussion, the space dimension $d$ satisfies $d\geq2$. However, in our
study we will focus on the case $d=2$ only.

In the system above, the scalar function $\rho=\rho(t,x)\geq0$ represents the density of the fluid, the vector field $u=u(t,x)\in\R^d$ its velocity field and
the scalar field $\Pi=\Pi(t,x)\in\R$ its pressure. The term $\nabla\Pi$ may be interpreted as a Lagrangian multiplier associated to
the divergence-free constraint on $u$.

We are interested on the initial-value problem for equations \eqref{eq:dd-E}.
Therefore, we supplement them with the initial datum 
\begin{equation} \label{eq:in-datum}
 \big(\rho,u\big)_{|t=0}\,=\,\big(\rho_0,u_0\big)\,,\qquad\qquad \mbox{ with }\qquad 
 \div u_0\,=\,0\,.
\end{equation}
Throughout this work, we assume \emph{absence of vacuum}. This means that there exist two positive constants $0<\rho_*\leq \rho^*$ such that
\begin{equation} \label{eq:vacuum}
 0<\rho_*\leq\rho_0\leq\rho^*\,.
\end{equation}

The goal of this paper is to establish \emph{geometric criteria} ensuring continuation, or rather blow-up, of regular solutions to equations \eqref{eq:dd-E}.
This means that, instead of looking for a control of the whole gradient of the velocity field, our criteria will require regularity only along a certain direction,
that is the one which is tangent to the level lines of the the density.
In order to motivate our study, let us given an overview of known results about the well-posedness theory and continuation criteria for system \eqref{eq:dd-E}.
We refer to Subsection \ref{ss:results} for the statement of our main results. 

\subsection{Review of the well-posedness theory} \label{ss:previous}

On the mathematical side, equations \eqref{eq:dd-E} retain, to some extent, the structure of a quasi-linear hyperbolic system. 
Of course, this is not completely true \tsl{strictu senso}, because (as is well-known in incompressible fluid dynamics)
the pressure gradient introduces non-local effects in the system.
However, at least under assumption \eqref{eq:vacuum} of absence of vacuum, the Euler system can be seen as a coupling of transport equations.
This observation hints to the fact that, for recovering well-posedness of equations \eqref{eq:dd-E}, it is natural to
consider functional spaces which are continuously embedded in the space $W^{1,\infty}$ of globally Lipschitz functions.

Without entering too much into the details, 
we only mention here results in the general framework of Besov spaces $B^s_{p,r}$, which includes the Sobolev class $H^s\equiv B^s_{2,2}$ and the H\"older
class $C^{k,\alpha}\equiv B^{k+\alpha}_{\infty,\infty}$ (with $k\in\N$ and $\alpha\in\,]0,1[\,$) as special cases.
Recall that the embedding $B^s_{p,r}\hookrightarrow W^{1,\infty}$ holds true if and only if the triplet of indices
$(s,p,r)\in\R\times[1,+\infty]\times[1,+\infty]$ satisfies the condition
\begin{equation} \label{cond:Lipschitz}
s\,>\,1\,+\,\frac{d}{p}\qquad\qquad\quad \mbox{ or }\qquad\qquad\quad s\,=\,1\,+\,\frac{d}{p}\quad \mbox{ and }\quad r\,=\,1\,.
\end{equation}
The case of Besov spaces with integrability index $1<p<+\infty$ was considered by Danchin in \cite{D1}.
The endpoint case $p=+\infty$ was covered slightly later, in \cite{D:F}, in both the finite energy and infinite energy situations:
respectively, either $u_0\in L^2$, or
$u_0\in L^p$ and $\nabla u_0\in L^q$ with $1/p+1/q\geq1/2$. Those conditions revealed to be important when dealing with
the pressure term, in order to guarantee the property $\nabla\Pi \in L^2$ (a property which seems somehow necessary in the study).
Paper \cite{F_2012} weakened the regularity assumptions by working in the context of
tangential (also dubbed ``striated'') regularity assumptions \tsl{\`a la Chemin} \cite{Ch_1991, Ch_1993, Ch_1995}.

We refer to the introduction of the above mentioned papers for a more complete overview of the literature concerning the well-posedness theory
for equations \eqref{eq:dd-E}.
A general well-posedness theorem, encompassing the results from \cite{D1, D:F}, can be stated in the following way
(this corresponds to Proposition 4.1 of \cite{Brav-F}, which we refer to also for the precise proof).

\begin{thm} \label{th:local}
Fix indices $(s,p,r)\in\R\times[1,+\infty]\times[1,+\infty]$ such that $p>1$ and one of the two conditions in \eqref{cond:Lipschitz}
is verified. Take an initial datum $\big(\rho_0,u_0\big)$ such that conditions \eqref{eq:in-datum} and \eqref{eq:vacuum} are satisfied, for two suitable constants
$0<\rho_*\leq\rho^*$. Assume in addition that $\nabla\rho_0\in B^{s-1}_{p,r}$ and that $u_0\in L^2\cap B^s_{p,r}$.

Then, there exist a positive time $T>0$ and a unique solution $\big(\rho,u,\nabla\Pi\big)$ to equations \eqref{eq:dd-E}-\eqref{eq:in-datum}-\eqref{eq:vacuum}
on $[0,T]\times\R^d$ such that:
\begin{enumerate}[(i)]
 \item $\rho\in \mc C_b\big([0,T]\times\R^d\big)$, with $\rho_*\leq\rho(t)\leq\rho^*$ for any $t\in[0,T]$ and $\nabla\rho\in\mc C_{\weak}\big([0,T];B^{s-1}_{p,r}\big)$;
 \item $u\in \mc C^1\big([0,T];L^2\big)\,\cap\,\mc C_\weak\big([0,T];B^{s}_{p,r}\big)$;
 \item $\nabla \Pi\in \mc C\big([0,T];L^2\big)\,\cap\,\mc C_\weak\big([0,T];B^{s}_{p,r}\big)$.
\end{enumerate}
The time continuity with values in Besov spaces holds with respect to the strong topology if $r<+\infty$.
\end{thm}

We observe that all the above mentioned well-posedness results hold true only \emph{locally in time}, 
even in specific case of space dimension $d=2$. This marks a severe difference between the non-homogeneous Euler equations \eqref{eq:dd-E} and its (classical)
homogeneous counterpart, obtained by setting $\rho\equiv1$, which has been known to be well-posed globally in time since the pioneering work of Wolibner \cite{Wol}
(see also classical textbooks like \cite{Maj-Bert} and \cite{BCD} for a review of global results in different functional frameworks).
As a matter of fact, when $\rho\equiv{\rm cst}$, the global existence of solutions strongly relies on the fact that, in two space dimensions,
the vorticity of the fluid, namely the quantity
\[
 \o\,:=\,\curl u\,=\,\d_1u^2\,-\,\d_2u^1\,,
\]
is simply transported by the velocity field. In the non-homogeneous setting, instead, the vorticity satisfies (still in $2$-D) the equation
\begin{equation} \label{eq:vort}
 \d_t\o\,+\,u\cdot\nabla\o\,+\,\nabla^\perp\left(\frac{1}{\rho}\right)\cdot\nabla\Pi\,=\,0\,.
\end{equation}
The notation $\perp$ represents the rotation of a $2$-D vector field of angle $\pi/2$; in particular, $\nabla^\perp\,=\,\big(-\d_2,\d_1\big)$.
In equation \eqref{eq:vort}, we notice the presence of the pressure gradient, which is somehow a quadratic term
(as $\nabla\Pi\approx |u|^2$ at the level of \tsl{a priori} estimates),
and, more importantly, of the gradient of the density, which grows (by transport theory) exponentially with the Lipschitz norm of $u$.

\subsection{Motivations: results related to the lifespan of solutions} \label{ss:lifespan}
The previous discussion shows that the local--in--time character of the (so-far settled) well-posedness theory is, to some extent, an effect of the heterogeneity.
In this context, we should nonetheless mention the recent preprint \cite{CWZZ}, which establishes global asymptotic stability around the Couette flow.
We also quote work \cite{F-Feir_2021}, which shows the existence of non-trivial steady solutions to system \eqref{eq:dd-E} with compact support in space,
and work \cite{Brav-F}, which proves global well-posedness for small data when a damping term is added to the momentum equation.

Apart from those works, dealing with quite particular cases, the issue of
global existence of regular solutions to equations \eqref{eq:dd-E} for general initial data still remains an open problem in the literature.
Therefore, attention has been devoted to the study of alternative questions related to the lifespan $T_{\ell ife}$ of solutions, mainly following two different directions:
the first one consists in providing a lower bound for $T_{\ell ife}$, the other one in giving blow-up/continuation criteria
in terms of lower order norms of the solutions.

In the first direction, we quote paper \cite{D:F}: there, the authors improved the lower bound which can be deduced
by classical quasi-linear hyperbolic theory and found that $T_{\ell ife}$ can be bounded from below by a function (roughly speaking, a double logarithm)
going to $+\infty$ when a suitable (critical) norm of $\nabla\rho_0$ goes to $0$. This result, which is specific
to the two-dimensional case, can be interpreted as an ``asymptotically global'' well-posedness result
(here, the asymptotics is for $\rho_0\to 1$) and is consistent with what is known for the homogeneous Euler equations.

In the second direction, instead, some continuation criteria appeared in \cite{D1, D:F}: they involved the integral in time of the Lipschitz norm of the velocity field,
plus a lower order term related to the pressure gradient. The latter term was first removed in paper \cite{Korean}, by requiring suitable integrability
conditions on $u_0$ and $\nabla u_0$; this result was later slightly generalised in \cite{Brav-F}. We report here below a general form of continuation criterion for
the Euler system \eqref{eq:dd-E}; this precise statement corresponds to Lemma 4.2 of \cite{Brav-F}.

\begin{prop} \label{p:cont:pri}
Let the assumptions of Theorem \ref{th:local} be in force.
In addition, assume that $p\geq 2$.
Let $\big(\rho,u,\nabla\Pi\big)$ be a local solution to system \eqref{eq:dd-E}-\eqref{eq:in-datum}-\eqref{eq:vacuum} related to the initial datum
$\big(\rho_0,u_0\big)$, defined on $[0,T^*[\,\times\R^d$ (for some time $T^*>0$) and such that the regularity properties (i)-(ii)-(iii)
stated in Theorem \ref{th:local} hold true for any $0<T<T^*$.

If $T^*<+\infty$ and
\begin{equation} \label{est:Du}
\int_0^{T^*} \left\|\nabla u(t) \right\|_{L^{\infty}}\,\dt\, <\, +\, \infty\,,
\end{equation} 
then $\big(\rho,u,\nabla\Pi\big)$ can be continued  beyond the time $T^*$ into a solution of \eqref{eq:dd-E}-\eqref{eq:in-datum}-\eqref{eq:vacuum}
with the same regularity.
\end{prop}

It is worth noticing that, in the above mentioned continuation criteria, nothing is really specific to the two-dimensional setting: those
criteria hold true in any space dimensions.
We also point out that an analogous of the celebrated Beale-Kato-Majda criterion \cite{BKM}, resting on the $L^\infty$ norm of the vorticity only, is \emph{not}
known in the non-homogeneous setting.

\medbreak
The scope of the present study is to give an alternative blow-up/continuation criterion for the density-dependent Euler equations \eqref{eq:dd-E},
which exhibit conditions only on certain derivatives of $u$ and not on its whole gradient. To some extent, this is analogue
to the Beale-Kato-Majda criterion, which, as already mentioned, consider only certain derivatives, namely the derivatives appearing in the definition of
the vorticity field of the fluid.

Here, we will formulate conditions on the derivatives aligned along a fixed direction. Thus, our approach is \emph{geometric} in nature.
At this point, it is important to mention that geometric blow-up conditions exist also for homogeneous fluids, see \cite{C-Feff} for the case of the $3$-D
Navier-Stokes equations, see \cite{C-Feff-M} for the Euler case. However, those works look more at the direction of the vorticity vector
and rest on its (H\"older) regularity. Our approach is different in spirit and, instead, is more related to Chemin's notion of tangential
regularity. 

Before going into the details of our study (see the next subsection), it is worth to point out that our results are
truly specific to the two-dimensional case.


\subsection{Main results} \label{ss:results}
The scope of this work is to replace condition \eqref{est:Du} by geometric regularity conditions on the velocity field.
Before stating our main results, some preliminary definitions are in order.

First of all, we introduce the type of \emph{geometric regularity} condition which plays a central role in this paper.
Assume that a (say) smooth vector field $X$ on $\R^d$ is given. Let $f:\R^d\longrightarrow \R$ be a Lipschitz function. We define the derivative of $f$ along $X$
as the function
\begin{equation*}
\d_Xf\,:=\,X\cdot\nabla f\,=\,\sum_{j=1}^dX^j\,\d_jf\,.
\end{equation*}
Notice that this quantity is well-defined pointwise almost everywhere on $\R^d$.
The definition naturally extends to time-dependent functions $f$ and vector fields $X$, as well as to the case in which the scalar field $f$ is replaced
by a vector field $F:\R^d\longrightarrow\R^d$.

Our results pertain to  the two-dimensional case only: they rely on a specific structure of the equations which arises in this situation. 
The case of higher space dimensions is out of the scopes of this work and will not be treated here. Therefore, hereafter we set $d=2$
and we define the space domain
\[
\Omega\,=\,\R^2\, 
\]
However, the adaptation of our results to the case of the two-dimensional torus $\T^2$ is also possible.

Next, for the sake of better readability of the statements, let us collect the basic assumptions which are common to all our results.

\paragraph{Basic assumptions.}
The first class of assumptions that we impose are needed simply in order to place our study in the framework of the well-posedness of equations \eqref{eq:dd-E};
more precisely, we want both Theorem \ref{th:local} and Proposition \ref{p:cont:pri} to hold.

As we will work in the general class of Besov spaces $B^s_{p,r}$ (this will be better justified in Subsection \ref{ss:ideas-proof}),
we need the embedding property $B^s_{p,r}\hookrightarrow W^{1,\infty}$, which is
guaranteed by either one of the two conditions in \eqref{cond:Lipschitz}. We will use the terminology \emph{subcritical case} and \emph{critical case} depending on
whether the Besov indices satisfy, respectively, the inequality condition
or the equality condition in \eqref{cond:Lipschitz}.
To sum up, we will impose the following assumption.
\begin{itemize}
 \item[\bf (A1)] The triplet $(s,p,r)\in \R\times[2,+\infty]\times[1,+\infty]$ is such that
\begin{align*}
&\mbox{either }\qquad\qquad s\,>\,1\,+\,\frac{2}{p}\qquad \mbox{ (subcritical case) }
\\
&\qquad \mbox{ or }\qquad\qquad s\,=\,1\,+\,\frac{2}{p}\qquad \mbox{ and }\qquad r\,=\,1 \qquad \mbox{ (critical case)}\,.
\end{align*}
\end{itemize}

Next, we take an initial datum in the above Besov regularity class and, in addition, such that properties \eqref{eq:in-datum} and \eqref{eq:vacuum} are fulfilled.
For the reasons explained in Subsection \ref{ss:previous}, we also need to impose a finite energy condition on the initial velocity field.
All this is stated in our second and third hypotheses.
\begin{itemize}
 \item[\bf (A2)] The initial datum $\big(\rho_0,u_0\big)$ verifies the properties
\[
\div u_0\,=\,0\qquad\qquad \mbox{ and }\qquad\qquad 0\,<\,\rho_*\,\leq\,\rho_0\,\leq\,\rho^*\,,
\]
for two suitable positive constants $\rho_*\leq\rho^*$.
 
\item[\bf (A3)] The initial datum $\big(\rho_0,u_0\big)$ is such that
\[
 u_0\,\in\,L^2(\R^2)\cap B^s_{p,r}(\R^2)\qquad\qquad \mbox{ and }\qquad\qquad \nabla\rho_0\,\in\,B^{s-1}_{p,r}(\R^2)\,.
\]
\end{itemize}

Finally, we need an extra assumption, of purely technical nature. It is necessary, within our approach, in order to prove our blow-up/continuation results.
\begin{itemize}
 \item[\bf (A4)] There exists a $p_0\in\,]2,+\infty[\,$ such that $\nabla u_0\,\in\,L^{p_0}(\R^2)$.
\end{itemize}
Observe that this last hypothesis is only needed when $p=+\infty$, whereas it comes for free in the case $p<+\infty$. Indeed,
when $2\leq p<+\infty$, by Besov embeddings one has that $\nabla u_0\in L^p(\R^2)\cap L^\infty(\R^d)$, whence the existence of the sought $p_0>2$
immediately follows.

\paragraph{Statement of the main results.}
After those preliminaries, we are ready to state the main result of this paper, namely the following blow-up criterion, which is geometric in nature.

\begin{thm} \label{th:blow-up}
Fix indices $(s,p,r)\in\R^3$ verifying assumption {\rm \tbf{(A1)}}.
Take an initial datum $\big(\rho_0,u_0\big)$ satisfying assumptions {\rm \tbf{(A2)}}-{\rm \tbf{(A3)}}-{\rm \tbf{(A4)}} and
let $\big(\rho,u,\nabla\Pi\big)$ be the related local solution to system \eqref{eq:dd-E}, as guaranteed by Theorem \ref{th:local}.
Define the time $T^*>0$ to be the lifespan of this solution.
Finally, define the vector field $X\,:=\,\nabla^\perp\rho$.

If $T^*\,<\,+\infty$, then:
\begin{itemize}
 \item in the subcritical case in which $s$ satisfies the strict inequality in \eqref{cond:Lipschitz}, one has
\[
\int^{T^*}_0\left\|\d_{X(t)}u(t)\right\|_{L^\infty}\,\dt\,=\,+\,\infty\,;
\]
\item in the critical case in which $s$ satisfies the equality in \eqref{cond:Lipschitz} (and then $r=1$), one has the property
\[
 \limsup_{t\to T^*}\left\|\d_{X(t)}u(t)\right\|_{B^{0}_{\infty,1}}\,=\,+\infty\,.
\]
\end{itemize}
\end{thm}

Some remarks about the previous result are in order. More comments will be formulated later on, after stating 
suitable counterparts of Theorem \ref{th:blow-up} related to the continuation of solutions.

\begin{rmk} \label{r:homog}
Observe that, as a byproduct of Theorem \ref{th:blow-up}, one recovers the global well-posedness of the classical (homogeneous) incompressible
Euler equations in $2$-D.
Indeed, in the case $\rho_*=\rho^*=\oline\rho$, the initial density $\rho_0\equiv\oline\rho$ is constant, which implies that
$\rho\equiv\oline\rho$ also at any later time.
In particular, one has $X\,=\,\nabla^\perp\rho\equiv0$ and the two conditions of the previous theorem are never satisfied. Hence, $T^*=+\infty$.

More in general, Theorem \ref{th:blow-up} gives rigorous confirmation to a rather intuitive fact: in $2$-D, singularities in equations \eqref{eq:dd-E}
cannot occur in regions of constant density.
\end{rmk}

\begin{rmk} \label{r:n-deriv}
The matrix $\nabla u$ contains only four elements. From Proposition \ref{p:cont:pri}, thanks to the divergence-free condition over $u$, we see that
we need a control on three of them to establish blow-up or continuation of solutions. Theorem \ref{th:blow-up} establishes that it is enough
to have information on only two derivatives with respect to a twisted (and possibly degenerate) direction.

Whether the above conditions can be replaced by a control only on a scalar quantity (like in the $2$-D version of the Beale-Kato-Maja criterion \cite{BKM})
remains open at present.
\end{rmk}

\begin{rmk} \label{r:delta}
In the critical case, Theorem \ref{th:blow-up} states (in case of finite lifespan) the blow-up of the $B^0_{\infty,1}$ norm of the quantity $\d_Xu$, pointwise in time.
Notice that, in this (critical) setting,
the well-posedness result of Theorem \ref{th:local} implies that $\nabla\rho$ and $\nabla u$ only belong to $B^0_{\infty,1}$;
from this, and using the fact that $\d_Xu\,=\,\div\big(X\otimes u\big)$, one can prove that the regularity of $\d_Xu$ is $B^0_{\infty,1}$ as well.

Hence, our blow-up criterion is really settled at critical regularity. In addition, and importantly,
it does not lose derivatives and, in this sense, it is an effective criterion (it would not be so, should it require a control on
higher regularity norms).
Its proof relies on very fine estimates, precisely because we are not allowed to lose any kind of regularity in the estimates. This is made possible by
the use of logarithmic Besov spaces.
\end{rmk}

The proof of Theorem \ref{th:blow-up} will be based on a contraposition argument: in fact, we will rather prove suitable continuation criteria
for regular solutions. 
For the sake of clarity, we present here the precise statements.
We start by dealing with the subcritical case.

\begin{thm} \label{th:cont-subcrit}
Fix indices $(s,p,r)\in\R^3$ verifying assumption {\rm \tbf{(A1)}} with strict inequality (subcritical case).
Take an initial datum $\big(\rho_0,u_0\big)$ satisfying assumptions {\rm \tbf{(A2)}}-{\rm \tbf{(A3)}}-{\rm \tbf{(A4)}} and
let $\big(\rho,u,\nabla\Pi\big)$ be the related local solution to system \eqref{eq:dd-E}, as guaranteed by Theorem \ref{th:local}.
Define the time $T^*>0$ to be the lifespan of this solution.
Finally, define the vector field $X\,:=\,\nabla^\perp\rho$.

Let $T>0$ be a given time and assume that
\begin{equation} \label{eq:cont_subcrit}
\int^T_0\left\|\d_{X(t)}u(t)\right\|_{L^\infty}\,\dt\,<\,+\infty\,.
\end{equation}
Then, the solution  $\big(\rho,u,\nabla\Pi\big)$ can be continued beyond the time $T$ into a solution to system \eqref{eq:dd-E}-\eqref{eq:in-datum}-\eqref{eq:vacuum}
having the same regularity.
\end{thm}

In the critical regularity setting, instead, we need to require a bound on stronger norms of $\d_Xu$, than its $L^\infty$ norm. This will be explained better
in Subsection \ref{ss:ideas-proof}. Our first continuation result in the critical framework reads as follows.

\begin{thm} \label{th:cont-crit_sum}
Fix indices $(s,p,r)\in\R^3$ verifying assumption {\rm \tbf{(A1)}} with equality (critical case).
Take an initial datum $\big(\rho_0,u_0\big)$ satisfying assumptions {\rm \tbf{(A2)}}-{\rm \tbf{(A3)}}-{\rm \tbf{(A4)}} and
let $\big(\rho,u,\nabla\Pi\big)$ be the related local solution to system \eqref{eq:dd-E}, as guaranteed by Theorem \ref{th:local}.
Define the time $T^*>0$ to be the lifespan of this solution.
Finally, define the vector field $X\,:=\,\nabla^\perp\rho$.

Let $T>0$ be a given time and assume that
\begin{equation} \label{eq:cont-cond_sum}
\int^T_0\left(\left\|\d_{X(t)}u(t)\right\|_{B^0_{\infty,1}}\,+\,\left\|\d_{X(t)}\big|u(t)\big|^2\right\|_{B^0_{\infty,1}}\right)\,\dt\,<\,+\infty\,.
\end{equation}
Then, the solution  $\big(\rho,u,\nabla\Pi\big)$ can be continued beyond the time $T$ into a solution to system \eqref{eq:dd-E}-\eqref{eq:in-datum}-\eqref{eq:vacuum}
having the same regularity.
\end{thm}

The two terms under the integral in \eqref{eq:cont-cond_sum} are
strictly related, as one has $\d_X|u|^2\,=\,2\,\d_Xu\cdot u$. Therefore, it is natural to try to control the second term of the sum by the first one
and formulate a continuation criterion in terms of the first term only.
This is non-trivial in the critical setting, because the Besov space $B^0_{\infty,1}$ is \emph{not} an algebra.
We manage to do that avoiding any loss of derivatives in the estimates; however, for technical reasons,
the integral in time must be replaced by a pointwise condition.

\begin{thm} \label{th:cont-crit}
Fix indices $(s,p,r)\in\R^3$ verifying assumption {\rm \tbf{(A1)}} with equality (critical case).
Take an initial datum $\big(\rho_0,u_0\big)$ satisfying assumptions {\rm \tbf{(A2)}}-{\rm \tbf{(A3)}}-{\rm \tbf{(A4)}} and
let $\big(\rho,u,\nabla\Pi\big)$ be the related local solution to system \eqref{eq:dd-E}, as guaranteed by Theorem \ref{th:local}.
Define the time $T^*>0$ to be the lifespan of this solution.
Finally, define the vector field $X\,:=\,\nabla^\perp\rho$.

Let $T>0$ be a given time and assume that
\begin{equation} \label{eq:cont-cond}
 \sup_{t\in[0,T[\,}\left\|\d_{X(t)}u(t)\right\|_{B^{0}_{\infty,1}}\,<\,+\infty\,.
\end{equation}
%
Then, the solution  $\big(\rho,u,\nabla\Pi\big)$ can be continued beyond the time $T$ into a solution to system \eqref{eq:dd-E}-\eqref{eq:in-datum}-\eqref{eq:vacuum}
having the same regularity.
\end{thm}

To conclude this part, let us formulate some remarks on the continuation conditions appearing in Theorems \ref{th:cont-subcrit}, \ref{th:cont-crit_sum} and
\ref{th:cont-crit}. The first one aims to relate our results with the already mentioned Beale-Kato-Majda continuation criterion 
for the incompressible Euler equations (see also Chapter 7 of \cite{BCD} and Chapter 4 of \cite{Maj-Bert} for more details and generalisations of it).

\begin{rmk} \label{r:geom}
The continuation criteria \eqref{eq:cont_subcrit}, \eqref{eq:cont-cond_sum} and \eqref{eq:cont-cond} are geometric in nature. Indeed,
they state that it is enough to control of $\nabla u$ only on a specific direction, namely the tangent direction to the level lines of the density.

To some extent those criteria, and especially the one of Theorem \ref{th:cont-subcrit}, can be seen as non-homogeneous counterparts
of the celebrated Beale-Kato-Majda continuation criterion \cite{BKM} for the (classical) homogeneous Euler equations, inasmuch as, also in that criterion,
it is enough to bound only certain combinations of the derivatives of $u$ (precisely, the ones appearing in the vorticity) in order to ensure
the continuation of the solution.

It should be noted, however, that the Beale-Kato-Majda criterion is linear, whereas ours is a \emph{non-linear} continuation criterion, which is rather unusual in
mathematical fluid mechanics.
\end{rmk}

The next remark, instead, makes a link between the geometric quantity investigated here and the notion of
\emph{striated regularity} \tsl{\`a la Chemin}, see \tsl{e.g.} \cite{Ch_1995}
(see also Chapter 7 of \cite{BCD} for more details and references about this notion and its generalisations). The propagation of that
kind of regularity for the density-dependent Euler system \eqref{eq:dd-E} was investigated in \cite{F_2012}.

\begin{rmk} \label{r:striated}
From the results of \cite{F_2012}, 
it may be possible to deduce that one can continue the solution as soon as the gradient of the vorticity is controlled (in H\"older norms)
along certain directions, which belong to a \emph{non-degenerate} family $\mc Y$ of vector fields and which are
transported by the flow associated to $u$.
When working with the velocity field, owing to Lemma 4.6 of \cite{F_2012}, this would be equivalent to require a bound for the $\mc C^\veps$ norm of $\d_Yu$,
for some $\veps>0$ and for any vector field $Y\in\mc Y$.

One must notice that the vector field $X=\nabla^\perp\rho$ is transported by the flow associated to $u$; hence, it can be chosen as a legitimate element
of the family $\mc Y$ mentioned above.
However, we see that our results improve the continuation criterion conjectured above (namely, the one based on propagation of $\mc C^\veps$ regularity for $\d_Yu$
for any $Y\in\mc Y$) in several aspects.
\begin{enumerate}[(i)]
\item  
We do not need the whole family of vector fields to ensure continuation.
Observe that ${\rm card}(\mc Y) \geq d+1$ and, when $d=2$, one typically needs three vector fields (see \tsl{e.g.} \cite{Ch_1995} and \cite{Paicu-Z}).
Instead, we only need to control $\nabla u$ along one precise direction, that is $X$.

\item 
Differently from the case of the family $\mc Y$, no kind of non-degeneracy conditions must be imposed on $X$;
for instance, $X$ needs not be non-vanishing.

\item 
We require a control on much lower regularity norms of $\d_Xu$, rather than $\mc C^\veps$ (which is the natural smoothness appearing
in the context of striated regularity, see \tsl{e.g.} \cite{Ch_1995} and \cite{F_2012}).
\end{enumerate}
\end{rmk}

As a last comment, we point out that it could be possible (actually, this is easy in the subcritical case) to replace, in the continuation criteria of
Theorems \ref{th:cont-subcrit} and \ref{th:cont-crit}, the geometric condition over $u$ by an analogous condition over the pressure gradient $\nabla \Pi$.
We refer to Remark \ref{r:pressure} below for more comments about this issue.

\subsection{Ideas of the proof} \label{ss:ideas-proof}

The proof of the main results relies on a novel approach to study regularity of solutions to the density-dependent incompressible Euler system \eqref{eq:dd-E}.
Such an approach consists in essentially two basic ingredients, which we are now going to present.

First of all, the regularity of the density is not analysed by looking at the transport equation for $\rho$, see the first equation in \eqref{eq:dd-E},
and its space derivatives. Instead, one (rather, and equivalently) looks at the evolution of the vector field $X\,:=\,\nabla^\perp\rho$. It turns out
that this vector field is \emph{transported} (in the sense of 1-forms) by the flow of the velocity field $u$, \tsl{i.e.} it satisfies the equation
\begin{equation} \label{intro_eq:X}
 \d_tX\,+\,u\cdot\nabla X\,=\,\d_Xu\,.
\end{equation}
A particular advantage of this equation, besides its semplicity, is that it makes immediately appear the quantity $\d_Xu$, which precisely encodes
the kind of geometric information we are interested in throughout this work.

The second ingredient consists in renouncing to propagate the regularity of the vorticity $\o\,=\,\d_1u^2\,-\,\d_2u^1$
of the flow by looking at its equation. As a matter of fact, in the non-homogeneous setting, the equation for $\o$ in $2$-D reads as in \eqref{eq:vort},
hence it looks complicated owing to the presence of the pressure gradient.
Instead, the idea is to introduce a new quantity $\eta$, defined as the ``vorticity'' of the momentum $m\,:=\,\rho\,u$, namely
\[
 \eta\,:=\,\d_1m^2\,-\,\d_2m^1\,,
\]
which allows us to get rid of the pressure from the equations. Indeed, from the momentum equation
one finds the relation satisfied by $\eta$:
\begin{equation} \label{intro_eq:eta}
 \d_t\eta\,+\,u\cdot\nabla\eta\,=\,\d_Xu\cdot u\,.
\end{equation}
This is a simple transport equation, where again the geometric information $\d_Xu$ comes into play.
Notice that the vector field $m$ is no more divergence-free, so one cannot recover it simply from $\eta$: an information on its divergence
is needed. However, we remark that
\[
 \div m\,=\,u\cdot\nabla\rho\,=\,-\,\d_t\rho\,,
\]
which is somehow a lower-order term, as no derivatives of $u$ come into play. Therefore, the key point really consists in exhibiting a control
on suitable norms of $\eta$.

With all these ingredients at hand, it is possible to establish (see Section \ref{s:prelim}) several \tsl{a priori} bounds for suitable Lebesgue norms
of the solution $\big(\rho,u,\nabla\Pi\big)$ to the non-homogeneous Euler system. It turns out that, thanks to all those bounds and to
a well-known logarithmic interpolation inequality, combined with an application of the Osgood lemma,
it is possible to give a rather direct proof of Theorem \ref{th:cont-subcrit}, dealing with the continuation criterion in the subcritical setting.

Let us now delve into the critical regularity framework.
In this case, one does not dispose anymore of the logarithmic interpolation inequality. Thus, the strategy consists in
using Proposition \ref{p:cont:pri} and directly estimating the $L^1_T(L^\infty)$ norm of the gradient of the velocity. For this, as already remarked above,
we want to avoid the use of the vorticity $\o$. Instead, we recover an expression for $\nabla u$ from the Helmholtz decomposition of the
momentum $m$. In particular, this expression involves the presence of two singular integral operators, which, as is well-known,
do not act continuously over $L^\infty$. In order to treat them in a low (as low as possible) regularity framework,
the natural choice is to work in the Besov space $B^0_{\infty,1}$, whence its presence in the statement of Theorems \ref{th:cont-crit_sum}
and \ref{th:cont-crit}.

We point out here that any other choice of the functional framework where studying the singular integrals
would result in the appearance of higher order regularity norms, which are out of control in the critical setting, thus leading to continuation criteria which
are not very intersting.

Let us go back to the bound of the singular integrals in the space $B^0_{\infty,1}$.
We notice that the simple transport structure of equations \eqref{intro_eq:X} and \eqref{intro_eq:eta} reveals to be crucial again, here:
indeed, one can take advantage of improved transport estimates \cite{Vis, HK}, which ensure bounds for the $B^0_{\infty,1}$ norms of $X$ and $\eta$,
with a growth which is only \emph{linear} in the Lipschitz norm of the transport field $u$.
This argument is enough to prove Theorem \ref{th:cont-crit_sum}: the norms appearing in \eqref{eq:cont-cond_sum} are precisely the right-hand sides
of equations \eqref{intro_eq:X} and \eqref{intro_eq:eta}.

As a finaly step to prove Theorem \ref{th:cont-crit}, one wants to bound the norm of $\d_Xu\cdot u$ in order to make only a condition on $\d_Xu$ appear.
Due to the critical regularity setting,
and especially to the fact that the space $B^0_{\infty,1}$ is not an algebra, this reveals to be hard. More precisely, bounding the remainder term
in a paraproduct decomposition of $\d_Xu\cdot u$ usually requires to work in positive regularity spaces, thus causing a loss of regularity.
We refer to Remark \ref{r:power} for more comments about that.
We also notice here that, as already pointed out in Remark \ref{r:delta}, in the critical setting a continuation criterion can be effective only
if it requires a control on norms which are \emph{at most} of order $B^0_{\infty,1}$, and \emph{not} higher.

The crucial point to make the $B^0_{\infty,1}$ norm of $\d_Xu$ appear despite the critical setting, is the use of the class of logarithmic Besov spaces.
This allows to produce a finer analysis, in particular regarding embeddings and interpolation, as stated in Lemma \ref{l:interp-log} and Proposition \ref{p:log-interp}
below. Thanks to those results, we are able to avoid any loss of derivatives on the geometric quantity and to prove the continuation criterion of Theorem
\ref{th:cont-crit} (the price to pay is to impose a pointwise condition in time, rather than an integral one).

To conlude, let us mention that the approach consisting in estimating the $L^1_T(L^\infty)$ norm of $\nabla u$
works whenever the Besov regularity indices verify one of the two properties in assumption \tbf{(A1)}. 
However, it gives rise to stronger conditions than the one appearing in \eqref{eq:cont_subcrit},
whence our interest in \eqref{eq:cont-cond_sum} and \eqref{eq:cont-cond} only in the critical setting.

\subsection*{Organisation of the paper}

The paper is organised in the following way.

In Section \ref{s:tools} we recall some basic tools, essentially coming from Littlewood-Paley theory, paradifferential calculus
and transport equations theory, which will be needed in our study.
We will develop our discussion in the context of \emph{logarithmic} Besov spaces,
which enable us to perform a finer analysis, required for treating the critical regularity framework.
Most of the material of this section is classical; however, it also contains the statement and proof of two refined interpolation inequalities
(which are crucial for our study) in the framework of logarithmic Besov spaces.

In Section \ref{s:prelim}, we establish some preliminary bounds for suitable Lebesgue norms of the solution. The crucial ingredient for getting those estimates
is the introduction of the new unknown $\eta$, as mentioned in Subsection \ref{ss:ideas-proof} above. At the same time, with those
bounds at hand, one is already in the position of proving Theorem \ref{th:cont-subcrit} (continuation criterion in the subcritical case).

The last section of this work is Section \ref{s:proof}, which contains the proof of Theorems \ref{th:cont-crit_sum} and \ref{th:cont-crit}
(continuation criteria in the critical case).
As can be extrapolated from the previous discussions, Theorem \ref{th:cont-crit} in fact implies Theorem \ref{th:cont-crit_sum};
therefore, we will mainly focus on the proof of the former result, pointing out in the course of the proof how to get the latter.
On the other hand, we have to notice that the continuation condition \eqref{eq:cont-cond} requires a control
on a $L^\infty$ norm in time, whereas conditions \eqref{eq:cont_subcrit} and \eqref{eq:cont-cond_sum} rest on $L^1$ norms in time. In the course of the proof, we will
strive as much as possible to rely on $L^1$-type conditions and highlight the precise point where the $L^\infty$ condition is needed.

\subsection*{Notation} 

Before going on, we fix here some pieces of notation which will be used throughout this paper.

For an interval $I\subset \R$ and a Banach space $\mf B$,
we denote by $\mc C\big(I;\mf B\big)$ the space of continuous bounded functions on $I$ with values in $\mf B$. For any $p\in[1,+\infty]$,
the symbol $L^p\big(I;\mf B\big)$ stands for the space of measurable functions on $I$ such that the map $t\mapsto \left\|f(t)\right\|_{\mf B}$ belongs to $L^p(I)$.
We also define $\mc C_b\big(I;\mf B\big)\,=\,\mc C\big(I;\mf B\big)\cap L^\infty\big(I;\mf B\big)$.
When $I=[0,T[\,$, very often we will use the notation $L^r_T(\mf B)\,=\,L^r\big([0,T[\,;\mf B\big)$.

In addition, let $\mc A$ be the predual of $\mf B$ and denote by $\lan\cdot,\cdot\ran_{\mf B\times \mc A}$ the duality pairing between $\mc A$
and $\mf B=\mc A^*$.
Then we set $\mc C_\weak\big(I;\mf B\big)$ the space of functions $f$ from $I$ to $\mf B$ which are weakly continuous, namely such that, for any $\phi\in \mc A$,
the map $t\mapsto \lan f(t),\phi \ran_{\mf B\times \mc A}$ is continuous over $I$.
We use the same notation for scalar-valued and vector-valued functions.

In order to make the writing easier and the reading ligther, in our estimates we will often avoid to write the explicit multiplicative constants
which allow to pass from one line to the other. In this case, we will write $A\,\lesssim\, B$
meaning that there exists a constant $C>0$, which may depend also on the quantities under control in the continuation criteria and on the functional norms of
the initial data such that $A\,\leq\,C\,B$.

Given two operators $\mc P$ and $\mc Q$, we will denote their commutator by the standard notation
\[
 \big[\mc P,\mc Q\big]\,:=\,\mc P\mc Q\,-\,\mc Q\mc P\,.
\]

Finally, given a two-dimensional vector field $v$, we define $v^\perp$ to be its rotation of angle $\pi/2$. More precisely, if $v\,=\,\big(v^1,v^2\big)$, then
$v^\perp\,=\,\big(-v^2,v^1\big)$. Similarly, we define the operator $\nabla^\perp$ as $\nabla^\perp\,=\,\big(-\d_2,\d_1\big)$.


\section*{Acknowledgements}

{\small

This work has been partially supported by the project CRISIS (ANR-20-CE40-0020-01), operated by the French National Research Agency (ANR),
by the Basque Government through the BERC 2022-2025 program and by the Spanish State Research Agency through the BCAM Severo Ochoa excellence accreditation
CEX2021-001142.
Finally, the first author also aknowledges the support of the European Union through the COFUND program [HORIZON-MSCA-2022-COFUND-101126600-SmartBRAIN3].

The author expresses his deep gratitude to J.-Y. Chemin and R. Danchin for very interesting and useful comments about a preliminary version of this work.

}

%

\section{Tools} \label{s:tools}

In this section, we collect several tools that are needed in our study.
We start by recalling some basic facts of Littlewood-Paley theory in $\R^d$, with $d\geq1$, and the definition of non-homogeneous Besov spaces.
Then, in Subsection \ref{ss:para} we introduce some elements of paradifferential calculus, with a special focus on paraproduct decomposition.
Finally, in Subsection \ref{ss:tools-est} we recall some well-known estimates for smooth solutions to transport equations.

Most of the material in this section is classical; details can be found \tsl{e.g.} in Chapter 2 of \cite{BCD}
(see also Chapters 4 and 5 of \cite{M-2008}).
However, since our analysis will require very fine estimates for treating the critical regularity setting,
we will develop our discussion in the framework of \emph{logarithmic Besov spaces},
as first introduced, to the best of our knowledge, in \cite{F_thesis} (see \cite{C-M} for the case of logarithmic Sobolev spaces). This will demand
some small adaptations of the classical theory, whose proofs can be found in \cite{F_thesis}-\cite{C-DS-F-M}, and, more importantly,
the proof of two new fine interpolation results, see Lemma \ref{l:interp-log} and Proposition \ref{p:log-interp} below.

Notice that transport estimates in the class of logarithmic Besov spaces have recently been investigated in \cite{Mey-Seis}. However, we will not need
those estimates in the present paper.

\subsection{Littlewood-Paley theory and logarithmic Besov spaces} \label{ss:LP}

In this subsection, we recall the main ideas of Littlewood-Paley theory in $\R^d$ and introduce the class of logarithmic Besov spaces.
For the sake of generality, we consider here the case of any space dimension $d\geq1$.

Consider a smooth radial function $\chi$ supported in the ball $\mc B(0,2)$ of center $0$ and radius $2$ in $\R^d$, with $\chi\equiv1$ in a neighborhood of $\mc B(0,1)$,
such that the map $r\mapsto\chi(r\,e)$ is nonincreasing over $\R_+$ for all unitary vectors $e\in\R^d$. Set
$\varphi\left(\xi\right)\,:=\,\chi\left(\xi\right)-\chi\left(2\xi\right)$ and
$\vphi_j(\xi)\,:=\,\vphi(2^{-j}\xi)$ for all $j\geq0$.
The dyadic blocks $(\Delta_j)_{j\in\Z}$ are defined by\footnote{Hereafter, we agree  that the notation  $f(D)$ stands for 
the pseudo-differential operator associated to the Fourier multiplier $f$, namely the operator $f(D):\,u\mapsto\mc{F}^{-1}[f(\xi)\,\what u(\xi)]$.} 
$$
\Delta_j\,:=\,0\quad\mbox{ if }\; j\leq-2,\qquad\Delta_{-1}\,:=\,\chi(D)\qquad\mbox{ and }\qquad
\Delta_j\,:=\,\varphi(2^{-j}D)\quad \mbox{ if }\;  j\geq0\,.
$$
For any $j\geq0$, the low frequency cut-off operators $S_j$ are defined 
as
\begin{equation} \label{eq:S_j}
S_j\,:=\,\chi(2^{-j}D)\,=\,\sum_{k\leq j-1}\Delta_{k}\qquad\mbox{ for }\qquad j\geq0\,.
\end{equation}
Notice that, for any integer $j$, $S_j$ and $\Delta_j$ are convolution operators by $L^1$ kernels whose norms are independent of the index $j$.
Thus, they act continuously from $L^p$ into itself, for any $p\in[1,+\infty]$.

Remark that the function $\chi$ can be chosen so that, for all $\xi\in\R^d$, one has the equality $\chi(\xi)+\sum_{j\geq0}\vphi_j(\xi)=1$.
Based on this partition of unity on the Fourier side, we obtain the so-called non-homogeneous \emph{Littlewood-Paley decomposition} of tempered distributions
on the ``physical space'' side:
\[
\forall\,u\,\in\,\mc S'(\R^d)\,,\qquad\qquad\qquad u\,=\,\sum_{j\geq-1}\Delta_ju\qquad \mbox{ in the sense of }\quad \mc S'(\R^d)\,.
\]

Next, we recall the so-called \emph{Bernstein inequalities}. They establish important properties linked with the action of derivatives on tempered
distributions with compact spectrum\footnote{The spectrum of a tempered distribution is defined as the support of its Fourier transform.}.

\begin{lemma} \label{l:bern}
Let  $0<r<R$.   A constant $C>0$ exists so that, for any integer $k\geq0$, any couple $(p,q)$ 
in $[1,+\infty]^2$, with  $p\leq q$,  and any function $u\in L^p$, for all $\lambda>0$ we have:
\begin{align*}
{\supp}\, \widehat u\, \subset\, \mc B(0,\lambda R)\,=\,\big\{\xi\in\R^d\,\big|\,|\xi|\leq\lambda R \big\}\qquad
\Longrightarrow\qquad
\|\nabla^k u\|_{L^q}\, \leq\,
 C^{k+1}\,\lambda^{k+d\left(\frac{1}{p}-\frac{1}{q}\right)}\,\|u\|_{L^p} \\[1ex]
{\supp}\, \widehat u   \, \subset\, \big\{\xi\in\R^d\,\big|\, r\lambda\leq|\xi|\leq R\lambda\big\}
\quad\Longrightarrow\quad C^{-k-1}\,\lambda^k\|u\|_{L^p}\,\leq\,
\|\nabla^k u\|_{L^p}\,
\leq\,C^{k+1} \, \lambda^k\|u\|_{L^p}\,.
\end{align*}
\end{lemma}

By use of Littlewood-Paley decomposition, we can define the class of Besov spaces. As already anticipated, however, we will include in our discussion
the case of possible logarithmic regularities. We refer to \cite{F_thesis}-\cite{C-DS-F-M} for more details on the logarithmic Besov spaces,
which themselves are a generalisation of the logarithmic Sobolev spaces first introduced in \cite{C-M}.

\begin{defin} \label{d:B}
%
Let $s$ and $\alpha$ be real numbers, and $1\leq p,r\leq+\infty$. The \emph{non-homogeneous logarithmic Besov space}
$B^{s+\alpha\log}_{p,r}$ is defined as the subset of tempered distributions $u$ for which
$$
\|u\|_{B^{s+\alpha\log}_{p,r}}\,:=\,
\left\|\Big(2^{js}\,(2+j)^{\alpha}\,\|\Delta_ju\|_{L^p}\Big)_{j\geq-1}\right\|_{\ell^r}\,<\,+\infty\,.
$$
\end{defin}

Of course, the previous definition extends the classical one of Besov spaces $B^s_{p,r}$, which correspond to the case $\alpha=0$.
For simplicity of notation, when $s=0$ we will simply write $B^{\alpha\log}_{p,r}$ instead of $B^{0+\alpha\log}_{p,r}$.

Recall that, for $\alpha=0$ and for any $k\in\N$ and~$p\in[1,+\infty]$,
we have the following chain of continuous embeddings:
$$
B^k_{p,1}\hookrightarrow W^{k,p}\hookrightarrow B^k_{p,\infty}\,,
$$
where  $W^{k,p}$ stands for the classical Sobolev space of $L^p$ functions with all the derivatives up to the order $k$ in $L^p$.
Logarithmic regularities help in making those interpolation properties more precise.
We also recall that, for all $s\in\R$, one has $B^s_{2,2}\equiv H^s$, and that,
for any $k\in\N$ and any $\veps\in\,]0,1[\,$, one has $B^{k+\veps}_{\infty,\infty}\equiv \mc C^{k,\veps}$, where $\mc C^{k,\veps}$ denotes the classical H\"older space.
It goes without saying that the previous equivalence relations entail also equivalence of the related norms.

As an immediate consequence of the first Bernstein inequality, one gets the following embedding result, in the context of logarithmic regularities. This
is a natural generalisation of the classical embedding properties of Besov spaces, which can be found \tsl{e.g.} in \cite{BCD} (see Proposition 2.71 therein).

\begin{prop} \label{p:log-emb}
The space $B^{s_1+\alpha_1\log}_{p_1,r_1}$ is continuously embedded in the space $B^{s_2+\alpha_2\log}_{p_2,r_2}$ whenever
$\,1\,\leq\,p_1\,\leq\,p_2\,\leq\,+\infty$ and one of the following conditions holds true:
\begin{itemize}
\item $s_2\,<\,s_1\,-\,d\,(1/p_1\,-\,1/p_2)\,$;
\item $s_2\,=\,s_1\,-\,d\,(1/p_1\,-\,1/p_2)\,$, $\,\alpha_2\,\leq\,\alpha_1\,$ and $\,1\,\leq\,r_1\,\leq\,r_2\,\leq\,+\infty\,$;
\item $s_2\,=\,s_1\,-\,d\,(1/p_1\,-\,1/p_2)\,$ and $\,\alpha_1\,-\,\alpha_2\,>\,1\,$.
\end{itemize}
\end{prop}

In particular, the previous proposition implies the following refined embedding property: for $\alpha>1$, one has
$B^{\alpha\log}_{\infty, \infty} \hookrightarrow B^0_{\infty, 1} \hookrightarrow L^\infty$. 
Actually, a simple computation allows to get the following generalisation of the previous chain of embeddings,
which will play a key role in the proof of Theorem \ref{th:cont-crit}:
\begin{equation} \label{emb:lg-B}
\forall\,r\in[1,+\infty]\,,\quad \forall\,\alpha>1-\frac{1}{r}\,,
\qquad\qquad 
B^{\alpha\log}_{\infty, r}\, \hookrightarrow\, B^0_{\infty, 1}\, \hookrightarrow\, L^\infty\,.
\end{equation}

Next, let us recall a classical logarithmic interpolation inequality, which corresponds to Proposition 2.104 of \cite{BCD}
(see also the subsequent Remark 2.105 therein).
\begin{prop} \label{p:interpol}
Fix some $\veps\in\,]0,1[\,$. There exists a constant $C>0$ such that, for any $f\in B^\veps_{\infty,\infty}$, one has the inequality
\[
\|f\|_{B^0_{\infty,1}}\,\leq\,\frac{C}{\veps}\,\|f\|_{B^0_{\infty,\infty}}\,
\left(1\,+\,\log\left(1\,+\,\frac{\|f\|_{B^\veps_{\infty,\infty}}}{\|f\|_{B^0_{\infty,\infty}}}\right)\right)\,.
\]
\end{prop}

Improvements of the above interpolation inequality are strictly necessary in the course of the proof to Theorem \ref{th:cont-crit}, in order to face the
critical regularity setting.
We present those improvements here below.
The crucial ingredient is to work in the setting of logarithmic Besov spaces, a fact which allows for a finer analysis.

The first result gives a refined quantitative version of the embeddings \eqref{emb:lg-B} and, at the same time,
improves Proposition \ref{p:interpol} by removing the logarithmic factor.
For simplicity of presentation, we focus on the case $r=+\infty$ only.

\begin{lemma} \label{l:interp-log}
Fix a real number $\alpha>1$. Then, there exists a universal constant $C>0$ such that, for any $f\in B^{\alpha\log}_{\infty,\infty}(\R^d)$,
one has
\[
 \|f\|_{B^{0}_{\infty,1}}\,\leq\,C\,\|f\|_{B^0_{\infty,\infty}}^{1-1/\alpha}\,\|f\|^{1/\alpha}_{B^{\alpha\log}_{\infty,\infty}}\,.
\]
\end{lemma}

\begin{proof}
We mimic the proof of Proposition \ref{p:interpol} from \cite{BCD}. More precisely, for any integer $N\geq1$ to be chosen later, we write
\begin{align*}
 \|f\|_{B^{0}_{\infty,1}}\,=\,\sum_{j\geq -1}\left\|\Delta_jf\right\|_{L^\infty}\,\leq\,\sum_{j= -1}^N\left\|\Delta_jf\right\|_{L^\infty}\,+\,
 \sum_{j\geq N+1}\left\|\Delta_jf\right\|_{L^\infty}\,.
\end{align*}
The first sum appearing on the right-hand side of the previous inequality can be bounded in the following simple way:
\[
 \sum_{j= -1}^N\left\|\Delta_jf\right\|_{L^\infty}\,\leq\,\left\|f\right\|_{B^0_{\infty,\infty}}\,(N+1)\,.
\]
For the second term, instead, we make the $B^{\alpha\log}_{\infty,\infty}$ norm appear by arguing as follows:
\begin{align*}
 \sum_{j\geq N+1}\left\|\Delta_jf\right\|_{L^\infty}\,&=\,\sum_{j\geq N+1}(j+2)^{\alpha}\,\left\|\Delta_jf\right\|_{L^\infty}\,(j+2)^{-\alpha} \\
 &\leq\,\left\|f\right\|_{B^{\alpha\log}_{\infty,\infty}}\,\sum_{j\geq N+1}(j+2)^{-\alpha}\,.
\end{align*}
The sum is of course convergence, as $\alpha>1$. Moreover, we can bound it by writing
\[
\sum_{j\geq N+1}(j+2)^{-\alpha}\,\leq\,\sum_{j\geq N}\int^{j+1}_j(\lambda+1)^{-\alpha}\,\dd\lambda\,=\,\int^{+\infty}_{N}(\lambda+1)^{-\alpha}\,\dd\lambda
\,\lesssim\,(N+1)^{-\alpha+1}\,.
\]

Putting everything together, in turn we find, for any $N\geq1$ fixed, the estimate
\begin{align*}
 \|f\|_{B^{0}_{\infty,1}}\,\lesssim\,\left\|f\right\|_{B^0_{\infty,\infty}}\,(N+1)\,+\,\left\|f\right\|_{B^{\alpha\log}_{\infty,\infty}}\,(N+1)^{-\alpha+1}\,.
\end{align*}
Optimising the choice of $N\geq 1$, we find
\[
(N+1)^{\alpha}\,\approx\,\frac{\left\|f\right\|_{B^{\alpha\log}_{\infty,\infty}}}{\left\|f\right\|_{B^0_{\infty,\infty}}}\,,
\]
which finally implies the sought interpolation inequality.
\end{proof}

The second result generalises the logarithmic interpolation inequality of Proposition \ref{p:interpol} to the logarithmic setting. The important point is that
this result allows to consider a large range of positive (logarithmic) regularities, even though we require the third Besov index $r$ to be $+\infty$.

\begin{prop} \label{p:log-interp}
Fix a real number $\alpha\in\big[0,3\log2\big]$. There exists a universal constant $C>0$ such that, for any Lipschitz function $f\in B^1_{\infty,\infty}(\R^d)$, one has
\[
\left\|f\right\|_{B^{\alpha\log}_{\infty,\infty}}\,\lesssim\,\left\|f\right\|_{B^{0}_{\infty,\infty}}\,
\left(1\,+\,\log\left(1\,+\,\frac{\left\|\nabla f\right\|_{B^{0}_{\infty,\infty}}}{\left\|f\right\|_{B^{0}_{\infty,\infty}}}\right)\right)^\alpha\,.
\]
\end{prop}

\begin{proof}
The case $\alpha=0$ being trivial, we restrict our attention to the case $\alpha>0$.
By definition of Besov norm, we have to estimate
\[
 \left\|f\right\|_{B^{\alpha\log}_{\infty,\infty}}\,=\,\sup_{j\geq-1}\Big(\big(2+j\big)^{\alpha}\,\left\|\Delta_jf\right\|_{L^\infty}\Big)\,.
\]
We argue in a similar way as before. 
Let us fix some $N\geq1$, whose precise value will be fixed later on. First of all, we can write
\begin{equation} \label{est:j-less-N}
\sup_{-1\leq j\leq N}\Big(\big(2+j\big)^\alpha\,\left\|\Delta_jf\right\|_{L^\infty}\Big)\,\lesssim\,\|f\|_{B^0_{\infty,\infty}}\,(2+N)^\alpha\,.
\end{equation}
On the other hand, by use of the second Bernstein inequality, we get
\begin{align*}
 \sup_{j\geq N+1}\Big(\big(2+j\big)^\alpha\,\left\|\Delta_jf\right\|_{L^\infty}\Big)\,&\lesssim\,
 \sup_{j\geq N+1}\Big(\big(2+j\big)^\alpha\,2^{-j}\,\left\|\Delta_j\nabla f\right\|_{L^\infty}\Big) \\
&\lesssim\,\left\|\nabla f\right\|_{B^0_{\infty,\infty}}\,\sup_{j\geq N+1}\Big(\big(2+j\big)^\alpha\,2^{-j}\Big)\,.
\end{align*}
Now, we use the fact that the function $h(z)=(2+z)^\alpha\,2^{-z}$ is decreasing for $z\geq \frac{\alpha}{\log 2}-2$. Notice that, for $\alpha\leq\log8$,
one has $\frac{\alpha}{\log 2}-2\leq 1\leq N$. Hence, we infer that
\begin{equation} \label{est:j-more-N}
 \sup_{j\geq N+1}\Big(\big(2+j\big)^\alpha\,\left\|\Delta_jf\right\|_{L^\infty}\Big)\,\lesssim\,(2+N)^\alpha\,2^{-N}\,\left\|\nabla f\right\|_{B^0_{\infty,\infty}}\,.
\end{equation}

Putting inequalities \eqref{est:j-less-N} and \eqref{est:j-more-N} and optimising in $N\geq 1$, we are led to perform the choice
\[
 N\,\approx\,\log_2\left(1\,+\,\frac{\left\|\nabla f\right\|_{B^0_{\infty,\infty}}}{\|f\|_{B^0_{\infty,\infty}}}\right)\,,
\]
which in turn yields the claimed inequality. The proposition is thus proven.
\end{proof}

\begin{rmk} \label{r:log-interp}
Generalisations in spirit of Proposition \ref{p:interpol}, which would make H\"older norms appear in place of the $B^1_{\infty,\infty}$
norm of $f$, are also possible. We do not pursue this issue here; besides, this would not yield any improvement in the proof of our main results.
\end{rmk}

\subsection{Paradifferential calculus} \label{ss:para}

We now apply Littlewood-Paley decomposition to state some useful results from paradifferential calculus. Again, we refer to Chapter 2
of \cite{BCD} for details and further results for the classical case of Besov spaces $B^s_{p,r}$, to \cite{F_thesis} and \cite{C-DS-F-M}
for the case of logarithmic Besov spaces (which we treat here).

In fact, in this work we only need properties linked to paraproduct decomposition.
So, to begin with, let us introduce the definition of the \emph{paraproduct operator}, after J.-M. Bony \cite{Bony}: given two tempered
distributions $u$ and $v$, the paraproduct operator $\mc T_uv$ is (formally) defined as
\[
 \mathcal{T}_uv\,:=\,\sum_jS_{j-1}u\,\Delta_j v\,.
\]
The importance of this definition relies on the observation that, formally, we can decompose the product $u\,v$ into the sum of three terms:
\begin{equation}\label{eq:bony}
u\,v\;=\;\mathcal{T}_uv\,+\,\mathcal{T}_vu\,+\,\mathcal{R}(u,v)\,,
\end{equation}
where the \emph{remainder} operator $\mc R(u,v)$ is defined by the formula
\[
\mathcal{R}(u,v)\,:=\,\sum_j\sum_{|k-j|\leq1}\Delta_j u\,\Delta_{k}v\,.
\]
The following propositions collect the main properties of the paraproduct and remainder operators when acting over logarithmic Besov spaces.
In order to understand better those statements, however, let us first recall the
(classical) properties of $\mc T_uv$ and $\mc R(u,v)$ in the context of the usual Besov spaces $B^s_{p,r}$.
\begin{prop}\label{p:op}
For any $(s,p,r)\in\R\times[1,+\infty]^2$ and $t>0$, the paraproduct operator 
$\mathcal{T}$ maps continuously $L^\infty\times B^s_{p,r}$ in $B^s_{p,r}$ and  $B^{-t}_{\infty,\infty}\times B^s_{p,r}$ in $B^{s-t}_{p,r}$.
Moreover, the following estimates hold:
\[
\|\mathcal{T}_uv\|_{B^s_{p,r}}\,\lesssim\,\|u\|_{L^\infty}\,\|\nabla v\|_{B^{s-1}_{p,r}}\qquad\mbox{ and }\qquad
\|\mathcal{T}_uv\|_{B^{s-t}_{p,r}}\,\lesssim\,\|u\|_{B^{-t}_{\infty,\infty}}\,\|\nabla v\|_{B^{s-1}_{p,r}}\,.
\]

For any $(s_1,p_1,r_1)$ and $(s_2,p_2,r_2)$ in $\R\times[1,+\infty]^2$ such that 
$s_1+s_2>0$, $1/p:=1/p_1+1/p_2\leq1$ and~$1/r:=1/r_1+1/r_2\leq1$,
the remainder operator $\mathcal{R}$ maps continuously~$B^{s_1}_{p_1,r_1}\times B^{s_2}_{p_2,r_2}$ into~$B^{s_1+s_2}_{p,r}$.
In the case $s_1+s_2=0$, if in addition $r=1$, the operator $\mathcal{R}$ is continuous from $B^{s_1}_{p_1,r_1}\times B^{s_2}_{p_2,r_2}$ with values
in $B^{0}_{p,\infty}$.
\end{prop}

We start by treating the paraproduct operator. The next statement collects its main continuity properties in the framework of logarithmic Besov spaces.
\begin{thm} \label{t:log-pp}
 Let $(s,\alpha,\beta)\,\in\R^3$ and $t>0$. Let also $(p,r,r_1,r_2)$ belong to $[1,+\infty]^4$.

The paraproduct operator 
$\mc T$ maps $L^\infty\times B^{s+\alpha\log}_{p,r}$ in $B^{s+\alpha\log}_{p,r}$,
and  $B^{-t+\beta\log}_{\infty,r_2}\times B^{s+\alpha\log}_{p,r_1}$ in $B^{(s-t)+(\alpha+\beta)\log}_{p,q}$,
with $1/q\,:=\,\min\left\{1\,,\,1/r_1\,+\,1/r_2\right\}$.
Moreover, the following estimates hold:
\begin{eqnarray*}
\|\mc T_uv\|_{B^{s+\alpha\log}_{p,r}} & \leq & C\,\|u\|_{L^\infty}\,\|\nabla v\|_{B^{(s-1)+\alpha\log}_{p,r}} \\
\|\mc T_uv\|_{B^{(s-t)+(\alpha+\beta)\log}_{p,q}} & \leq &
C\,\|u\|_{B^{-t+\beta\log}_{\infty,r_2}}\,\|\nabla v\|_{B^{(s-1)+\alpha\log}_{p,r_1}}\,.
\end{eqnarray*}
Moreover, the second inequality still holds true if $t=0$, when $\beta\leq0$ and $r_2=+\infty$.
\end{thm}

The next result, instead, collects the continuity properties of the remainder operator acting on logarithmic Besov spaces.

\begin{thm} \label{t:log-r}
 Let $(s,t,\alpha,\beta)\in\R^4$ and $(p_1,p_2,r_1,r_2)\in[1,+\infty]^4$ be such that
\[
\frac{1}{p}\,:=\,\frac{1}{p_1}\,+\,\frac{1}{p_2}\,\leq\,1\qquad\mbox{ and }\qquad
\frac{1}{r}\,:=\,\frac{1}{r_1}\,+\,\frac{1}{r_2}\,\leq\,1\,.
\]
There exists a suitable constant $C>0$ such that the following facts hold true.
\begin{itemize}
 \item[(i)] If $s+t>0$, then for  any
$(u,v)\in B^{s+\alpha\log}_{p_1,r_1}\times B^{t+\beta\log}_{p_2,r_2}$ we have
\[
\left\|R(u,v)\right\|_{B^{(s+t)+(\alpha+\beta)\log}_{p,r}}\;\leq\;C\,\|u\|_{B^{s+\alpha\log}_{p_1,r_1}}\,
\|v\|_{B^{t+\beta\log}_{p_2,r_2}}\,.
\]
\item[(ii)] If $s+t=0$, $\alpha+\beta\geq0$ and $r=1$, then for any 
$(u,v)\in B^{s+\alpha\log}_{p_1,r_1}\times B^{t+\beta\log}_{p_2,r_2}$
one has the inequality
\[
\left\|R(u,v)\right\|_{B^{(\alpha+\beta)\log}_{p,\infty}}\;\leq\;C\,\|u\|_{B^{s+\alpha\log}_{p_1,r_1}}\,
\|v\|_{B^{t+\beta\log}_{p_2,r_2}}\,.
\]
\end{itemize}
\end{thm}

To conclude this part, let us recall 
the so-called \emph{tame estimates} for the product of two tempered distributions in  $B^s_{p,r}\cap L^\infty$, when $s>0$.

\begin{prop}\label{p:tame}
Let $(s,p,r)\in\R\times[1,+\infty]\times[1,+\infty]$ be such that that $s > 0$. Then, for any $f$ and $g$ belonging to $L^\infty\cap B^s_{p,r}$, the product
$fg$ also belongs to that space and we have
\begin{equation}
\label{alg:prop:2}
\left\| f\,g \right\|_{B^s_{p,r}}\, \lesssim \,\| f \|_{L^\infty}\, \|g\|_{B^s_{p,r}}\, +\, \| f \|_{B^s_{p,r}} \,\| g \|_{L^\infty}\,.
\end{equation}
\end{prop}

\subsection{Transport estimates in Besov spaces} \label{ss:tools-est}

System \eqref{eq:dd-E} has essentially the structure of a coupling of transport equations, at least under condition \eqref{eq:vacuum}.
As a matter of fact, this transport structure will reflect also in the equations satisfied by the other quantities $X$ and $\eta$ we will consider,
recall relations \eqref{intro_eq:X} and \eqref{intro_eq:eta} of the Introduction.

Therefore, let us recall basic \tsl{a priori} estimates for smooth solutions to the linear transport equation
\begin{equation}\label{eq:TV}
\left\{\begin{array}{l}
\partial_t f\, +\, v \cdot \nabla f\, = \,g \\[1.5ex]
f_{|t = 0}\, =\, f_0\,.
\end{array}\right.
\end{equation}
Throughout this part, the velocity field $v=v(t,x)$ will always be assumed to be divergence-free, \tsl{i.e.} $\div v = 0$,
and Lipschitz continuous with respect to the space variable.

First of all, since $v$ is assumed to be smooth and divergence-free, classical results ensure us that, for any $p\in[1,+\infty]$ and any time $t\geq0$ for which
$v$ and $g$ are well defined, one has
\begin{equation} \label{est:transp-L}
 \left\|f(t)\right\|_{L^p}\,\leq\,\left\|f_0\right\|_{L^p}\,+\,\int^t_0\left\|g(\t)\right\|_{L^p}\,\dd\t\,.
\end{equation}

The following statement, instead, pertains to the study of the transport equation \eqref{eq:TV} in the class of Besov spaces. It
contains the conclusions of Theorems 3.14 and 3.19 of \cite{BCD}.
\begin{thm}\label{th:transport}
Let $(s,p,r)\in\R\times[1,+\infty]^2$ satisfy the Lipschitz condition \eqref{cond:Lipschitz} and $T>0$ be fixed.
Assume that $v$ is a divergence-free vector field belonging to $L^1\big([0,T];B^s_{p,r}\big)$ such that, for some
$q > 1$ and $M > 0$, $v \in L^q\big([0,T];B^{-M}_{\infty, \infty}\big)$.
Finally, let $\s\in\R$ be such that
\[
\s\,\geq\,-\,d\,\min\left\{\frac{1}{p}\,,\,\frac{1}{p'}\right\}\qquad\qquad \mbox{ or, \ \ \ \ \  if }\ \ \div v=0\,,\quad 
\s\,\geq\,-\,1\,-\,d\,\min\left\{\frac{1}{p}\,,\,\frac{1}{p'}\right\}\,.
\]

Then, for any external force $g \in L^1\big([0,T];B^\s_{p,r}\big)$ and any initial datum $f_0 \in B^\s_{p,r}$, the transport equation \eqref{eq:TV} has a unique solution $f$ in the space:
\begin{itemize}
\item $\mc C\big([0,T];B^\s_{p,r}\big)$ if $r < +\infty$;
\item $\left( \bigcap_{\s'<\s} \mc C\big([0,T];B^{\s'}_{p,\infty}\big) \right) \cap \mc C_{\weak}\big([0,T];B^\s_{p, \infty}\big)$, in the case $r = +\infty$.
\end{itemize}
Moreover, after defining the function $V(t)$ as
\[
V(t)\,:=\,\int^t_0\left\|\nabla v(\t)\right\|_{L^\infty\cap B^{s-1}_{p,r}}\,\dd\t\,,
\]
the unique solution $f$ satisfies the following estimate, for a suitable universal constant $C>0$:
\begin{equation*} 
\forall\,t\in[0,T]\,,\qquad
\| f(t) \|_{B^\s_{p,r}}\, \leq\, e^{C\,V(t)}\,\left(\| f_0 \|_{B^\s_{p,r}} + \int_0^t e^{-C\,V(\tau)}\, \| g(\t) \|_{B^\s_{p,r}} \dd\t\right)\,.
\end{equation*}
In the case when $v=f$, the previous estimate holds true with $V'(t)\,=\,\left\|\nabla f(t)\right\|_{L^\infty}$.
\end{thm}

The exponential growth, which appears in the previous estimate, of the Besov norm in terms of the Lipschitz norm of $v$ is very often dramatic
in applications to fluid mechanics, because, generally speaking, it is responsible for only local well-posedness results.
However, as discovered by Vishik in \cite{Vis} (see also \cite{HK} by Hmidi and Keraani for a different proof), the estimates improve
when working in spaces with $0$ index of regularity, exhibiting a behaviour closer to the $L^p$ case \eqref{est:transp-L}:
the $B^0_{p,r}$ norm of the solution grows only \emph{linearly} with the Lipschitz norm of the transport field $v$.
We refer to Theorem 3.18 of \cite{BCD} for the statement presented here.

\begin{thm}\label{th:B^0}
Assume that the velocity field $v$ satisfies $\nabla v \in L^1_T(L^\infty)$ and $\div v=0$. Let $(p,r) \in [1, +\infty]^2$.

Then there exists a constant $C = C(d)>0$, only depending on the space dimension $d\geq2$, such that, for any solution $f$ to problem \eqref{eq:TV} in
$\mc C\big([0,T];B^0_{p,r}\big)$, or in $\mc C_w\big([0,T];B^0_{p,\infty}\big)$ if $r=+\infty$, we have
\[ 
\| f \|_{L^\infty_T(B^0_{p, r})}\, \leq\, C\, \left( \| f_0 \|_{B^0_{p, r}}\, +\, \| g \|_{L^1_T(B^0_{p, r})}\right)\;
\left( 1+\int_0^T\| \nabla v(\tau) \|_{L^\infty}\,{\rm d} \tau \right)\,.
\] 
\end{thm}

\section{Preliminary bounds and proof of Theorem \ref{th:cont-subcrit}} \label{s:prelim}

In this section, we collect several preliminary informations which are needed in the proof of our results.
First of all, we recall basic estimates coming from the energy conservation and the transport structure of the mass conservation equation.
After that, we exhibit a uniform bound for
$\nabla\rho$ in $L^\infty$. Next, 
we introduce a new unknown $\eta$, which has the advantage of making the pressure disappear from the equations and, at the same time,
highlighting the geometric quantity we are interested in. 
Finally, we will use this new unknown to control the $L^\infty$ norm of $u$.

Those bounds are all what is needed to prove Theorem \ref{th:cont-subcrit}. Therefore, the last part of this section will be devoted precisely
to the proof of that result. In particular, this will also validate the first blow-up criterion of Theorem \ref{th:blow-up}.

Throughout this section, we are going to work with assumption \eqref{eq:cont_subcrit}, that we recall here for the reader's convenience:
\begin{equation} \label{assump:weaker}
K_0\,:=\,\int^T_0\left\|\d_Xu(t)\right\|_{L^\infty}\,\dt\,<\,+\,\infty\,.
\end{equation}
Notice that condition \eqref{assump:weaker} is a weaker assumption than \eqref{eq:cont-cond_sum} and \eqref{eq:cont-cond}.
Hence, the bounds established in this section will be valid also when proving Theorems \ref{th:cont-crit_sum} and \ref{th:cont-crit}.

\subsection{Basic estimates for the density and the velocity field} \label{ss:energy}

Here we show estimates for suitable $L^q$ norms of the density $\rho$ and the velocity field $u$.

First of all, thanks to assumption \tbf{(A1)} and the divergence-free condition over $u$,
from the transport equation satisfied by $\rho$ one immediately obtains that
\begin{equation} \label{est:rho-inf}
\forall\,(t,x)\in[0,T[\,\times\R^2\,,\qquad\qquad \rho_*\,\leq\,\rho(t,x)\,\leq\,\rho^*\,.
\end{equation}

Next, we perform an energy estimate on the momentum equation, namely the second equation appearing in system \eqref{eq:dd-E}.
Since the velocity field $u$ is divergence-free, the pressure term identically vanishes when multiplied in $L^2$ by $u$. Standard computations thus yield
\[
\frac{\dd}{\dd t}\int\rho\,|u|^2\,\dd x\,=\,0\,,
\]
which implies, after an integration in time, the following estimate: for all $t\in[0,T[\,$, one has
\[ 
\left\|\sqrt{\rho(t)}u(t)\right\|_{L^2}^2\,=\,\left\|\sqrt{\rho_0}u_0\right\|_{L^2}^2\,.
\] 
Using \eqref{est:rho-inf}, we immediately deduce that
\begin{equation} \label{est:u-2}
\sup_{t\in[0,T[\,}\left\|u(t)\right\|_{L^2}\,\lesssim\,\left\|u_0\right\|_{L^2}\,,
\end{equation}
where the (implicit) multiplicative constant depends on $\rho_*$ and $\rho^*$ only. 

From now on the (often implicit) constants appearing in the estimates may depend not only on $\rho_*$ and $\rho^*$,
but also on $K_0$, defined in \eqref{assump:weaker}, and on suitable functional norms of the initial data $\rho_0$ and $u_0$. For the sake of better readability,
we state this fact now and for all, and we will avoid to point it out at every time.

\subsection{A uniform bound for the gradient of the density} \label{ss:D-rho}
Recall that, here and throughout this work, we have set
\[
 X\,:=\,\nabla^\perp\rho\,,
\]
where $\nabla^\perp\,=\,\big(-\d_2,\d_1\big)$ and, more in general, the $\perp$ operation consists in the rotation of a $2$-D vector of angle $\pi/2$.

With the previous convention, straightforward computations show that
the vector field $X$ is transported (in the sense of $1$-forms) by the flow of $u$, namely it solves
\begin{equation} \label{eq:X}
 \d_tX\,+\,u\cdot\nabla X\,=\,\d_Xu\,.
\end{equation}
In particular, this equation implies that the following commutator relation holds true:
\begin{equation} \label{eq:comm}
\big[\d_X\,,\,\d_t\,+\,u\cdot\nabla\big]\,=\,0\,.
\end{equation}

From classical transport estimates in Lebesgue spaces applied to equation \eqref{eq:X}, we deduce the following bound for $\|X\|_{L^\infty}$,
whence for $\|\nabla\rho\|_{L^\infty}$: for any $t\geq0$, one has
\[ 
\left\|\nabla\rho(t)\right\|_{L^\infty}\,\leq\,\left\|\nabla\rho_0\right\|_{L^\infty}\,+\,\int^t_0\left\|\d_Xu(\t)\right\|_{L^\infty}\,\dd\tau\,.
\] 
Therefore, in view of assumption \eqref{assump:weaker}, we infer that
\begin{equation} \label{est:unif_D-rho}
\sup_{t\in[0,T[\,}\left\|\nabla\rho(t)\right\|_{L^\infty}\,<\,+\,\infty.
\end{equation}

\subsection{The ``vorticity'' of the momentum: the new unknown $\eta$} \label{ss:eta}

In this subsection, we introduce a new unknown, which will play a key role in our proof.
More precisely, given $m\,:=\,\rho\,u$ the momentum of the fluid, we define the scalar field $\eta$ as 
\[
 \eta\,:=\,\curl m\,=\,\d_1\big(m^2\big)\,-\,\d_2\big(m^1\big)\,.
\]
We observe (see the details in \cite{F-GB-S}, for instance) that $\eta$ solves the transport equation
\begin{equation} \label{eq:eta}
 \d_t\eta\,+\,u\cdot\nabla\eta\,=\,\d_Xu\cdot u\,,
\end{equation}
where we notice that the following equality holds true:
\[
 \d_Xu\cdot u \,=\,\big(\nabla^\perp\rho\cdot\nabla\big) u\cdot u\,=\,\frac{1}{2}\,\nabla^\perp\rho\cdot\nabla|u|^2\,=\,\frac{1}{2}\,\d_X|u|^2\,.
\]

Now, we can perform $L^q$ estimates for the transport equation \eqref{eq:eta}. For reasons which are going to be clear in a while,
we consider $q\in[p_0,+\infty]$ (in fact, $q=p_0$ and $q=+\infty$ would be enough). Recall that the index $p_0$
has been fixed in assumption \tbf{(A4)} and verifies $p_0\in\,]2,+\infty[\,$.
Thus, by transport estimates \eqref{est:transp-L} we infer that, for any $t\in[0,T[\,$, one has
\begin{equation} \label{est:eta-L_first}
\left\|\eta(t)\right\|_{L^q}\,\leq\,\left\|\eta_0\right\|_{L^q}\,+\,\int^t_0\left\|\d_X u(\t)\right\|_{L^\infty}\,\left\|u(\t)\right\|_{L^q}\,\dt\,,
\end{equation}
where we have defined $\eta_0\,:=\,\curl\big(\rho_0\,u_0\big)$.
Observe that, owing to the equality
\begin{equation} \label{eq:eta-vort}
 \eta\,=\,\rho\,\omega\,+\,u\cdot\nabla^\perp\rho\,,
\end{equation}
from assumptions \tbf{(A1)} and \tbf{(A3)} and from the condition $p_0>2$, we easily deduce that,
for any $q\in[p_0,+\infty]$, one has
\[
\left\|\eta_0\right\|_{L^q}\,\leq\,\rho^*\,\left\|\nabla u_0\right\|_{L^q}\,+\,\left\|u_0\right\|_{L^2\cap L^\infty}\,\left\|\nabla\rho_0\right\|_{L^\infty}
\,<\,+\infty\,.
\]
From now on, we are going to mainly work with $q = p_0$; in only one point of the proof, we will need to take also $q=+\infty$ in the previous estimates.

\subsection{Bounds on the Lebesgue norms of $u$ and $\eta$} \label{ss:u-Leb}

Recall that the index $p_0$ has been defined in assumption \tbf{(A4)}.
Taking inequality \eqref{est:eta-L_first} into account, where we fix the choice $q=p_0$, it is apparent that we need to find a bound for
the velocity field $u$ in the space $L^\infty_T(L^{p_0})$.
Owing to \eqref{est:u-2}, it is enough to bound its $L^\infty$ norm (both in time and space).
For this, we will use the new unknown $\eta$.

Notice that the momentum $m\,:=\,\rho\,u$ is \emph{not} divergence-free. Therefore, its Helmholtz decomposition writes as
\begin{equation} \label{eq:Helm-mom}
m\,=\,\rho\,u\,=\,-\,\nabla^\perp(-\Delta)^{-1}\eta\,-\,\nabla(-\Delta)^{-1}\big(u\cdot\nabla\rho\big)\,.
\end{equation}
Thanks to this relation and to the fact that $p_0>2=d$, for any $t\in[0,T[\,$ fixed (which we omit from the notation), we can bound
\begin{align*}
 \left\|m\right\|_{L^\infty}\,&\lesssim\,\left\|m\right\|_{W^{1,p_0}}\,=\,\left\|\rho\,u\right\|_{W^{1,p_0}} \\
 &\lesssim\,\left\|\rho\,u\right\|_{L^{p_0}}\,+\,\left\|\eta\right\|_{L^{p_0}}\,+\,\left\|u\cdot\nabla\rho\right\|_{L^{p_0}}\,,
\end{align*}
where we have also used the Calder\'on-Zygmund theory to bound the $L^{p_0}$ norm of $\nabla m$ by the sum of the $L^{p_0}$ norms of $\eta$ and $u\cdot\nabla \rho$.
At this point, the H\"older inequality and an interpolation argument yield
\begin{align*}
\left\|m\right\|_{L^\infty}\,&\lesssim\,\rho^*\,\left\|u\right\|_{L^{p_0}}\,+\,
\left\|\eta\right\|_{L^{p_0}}\,+\,\left\|u\right\|_{L^{p_0}}\,\left\|\nabla\rho\right\|_{L^\infty} \\
&\lesssim\,\left\|u\right\|^\theta_{L^2}\,\left\|u\right\|^{1-\theta}_{L^\infty}\,+\,\left\|\eta\right\|_{L^{p_0}}\,.
\end{align*}
Here, $\theta\,=\,2/p_0\,\in\;]0,1[\,$; in addition, in passing from the first inequality to the second one, we have made use of
the uniform bound established in \eqref{est:unif_D-rho}.

At this point,  we observe that, owing to \eqref{est:rho-inf}, we have
\[
\left\|u\right\|_{L^\infty}\,\lesssim\, \left\|m\right\|_{L^\infty}\,.
\]
Hence, from the above estimate, inequality \eqref{est:u-2} and an application of the Young inequality, we immediately infer that
\begin{equation} \label{est:u-inf_prelim}
\left\|u\right\|_{L^\infty}\,\lesssim\,1\,+\,\left\|\eta\right\|_{L^{p_0}}\,.
\end{equation}
In particular, the previous bound implies that
\[
 \left\|u\right\|_{L^{p_0}}\,\lesssim\,\left\|u\right\|_{L^2\cap L^\infty}\,\lesssim\,1\,+\,\left\|\eta\right\|_{L^{p_0}}\,.
\]

Now, we insert this last inequality into \eqref{est:eta-L_first}, where we take $q=p_0$ first and, then, $q=+\infty$. From an application of Gr\"onwall's lemma and
assumption \eqref{assump:weaker}, we deduce that
\begin{equation} \label{est:eta_Leb}
 \sup_{t\in[0,T[\,}\left\|\eta(t)\right\|_{L^{p_0}\cap L^\infty}\,<\,+\,\infty\,.
\end{equation}
Going back to \eqref{est:u-inf_prelim}, in turn we infer that
\begin{equation} \label{est:u-inf}
  \sup_{t\in[0,T[\,}\left\|u(t)\right\|_{L^\infty}\,<\,+\,\infty\,.
\end{equation}

We conclude this part by presenting a family of interpolation inequalities, which will be very useful later to bound higher order norms
of the velocity field. We start by observing that, in space dimension $d=2$, one has the continuous embedding
$L^2\hookrightarrow B^{-1}_{\infty,\infty}$; at the same time, one also has (in any space dimension) $W^{1,\infty}\hookrightarrow B^1_{\infty,\infty}$.
Therefore, for any functional space $\mc Z$ such that $B^1_{\infty,\infty}\hookrightarrow \mc Z \hookrightarrow B^{-1}_{\infty,\infty}$,
one can write, for a suitable $\theta\in\,]0,1[\,$, the interpolation inequality
\[
 \|u\|_{\mc Z}\,\lesssim\,\|u\|_{L^2}^\theta\,\|u\|_{W^{1,\infty}}^{1-\theta}\,.
\]
By use of the uniform bounds \eqref{est:u-2} and \eqref{est:u-inf} and an application of the Young inequality,
we then find that, 
for any time $t\in[0,T[\,$, the following inequality holds true:
\begin{align} \label{est:interp-u}
 \left\|u(t)\right\|_{\mc Z}\,&\leq\,C_\veps\,+\,\veps\,\left\|\nabla u(t)\right\|_{L^\infty}\,,
\end{align}
where $\veps>0$ can be chosen as small as needed, up to increase accordingly the value of the numerical constant $C_\veps$.
In the next computations, we will need to apply inequality \eqref{est:interp-u} with $\mc Z=B^0_{\infty,1}$ and $\mc Z=B^{\alpha}_{\infty,\infty}$,
for $\alpha\in\,]0,1[\,$ small.

\subsection{The subcritical case: proof of Theorem \ref{th:cont-subcrit}} \label{ss:subcrit-proof}

Thanks to the bounds established above, we can give the proof of Theorem \ref{th:cont-subcrit}, hence also of the first part of Theorem \ref{th:blow-up}.
Therefore, in this subsection we will assume to be in the subcritical case, meaning that the regularity index $s$ verifies $s>1+d/p$ in \eqref{cond:Lipschitz}.
In addition, we assume that condition \eqref{assump:weaker} holds true.

Our approach does not uses Proposition \ref{p:cont:pri} (which will be exploited in the critical case, instead). However, we will follow the main steps
of its proof, as given in \cite{Brav-F},
and show that one can bound uniformly on $[0,T[\,$ the norm of the solution, namely the quantity
\[
 E(t)\,:=\,\left\|\rho\right\|_{L^\infty}\,+\,\left\|\nabla\rho(t)\right\|_{B^{s-1}_{p,r}}\,+\,\left\|u(t)\right\|_{L^2\cap B^s_{p,r}}\,+\,
\left\|\nabla\Pi(t)\right\|_{L^2\cap B^s_{p,r}}\,,
\]
by taking only condition \eqref{assump:weaker} for granted. 

We now claim that, instead of $E(t)$, it is enough to prove that the quantity
\[
  \mc N(t)\,:=\,\left\|\nabla\rho(t)\right\|_{B^{s-1}_{p,r}}\,+\,\left\|u(t)\right\|_{B^s_{p,r}}
\]
remains bounded on the time interval $[0,T[\,$, under condition \eqref{assump:weaker}. Indeed, on the one hand
the lower order norms of $\rho$ and $u$ are already known to be bounded, thanks to \eqref{est:rho-inf} and \eqref{est:u-2} above; on the other hand,
the pressure can be estimated in terms of the norms of $\rho$ and $u$, as it can be deduced from the next two statements.

The first result is a well-known interpolation inequality, allowing us to estimate the Lipschitz norm of $u$ by the $L^\infty$ norm of
the vorticity $\o$ times a logarithmic factor depending on the higher order norms. In our context, this yields the following result.
\begin{lemma} \label{l:Du-log-interp}
Let the assumptions of Theorem \ref{th:cont-subcrit} be in force. Then, there exists a constant $C>0$ such that, for any
fixed time $t\in[0,T[\,$, one has the inequality
\[
 \left\|\nabla u(t)\right\|_{L^\infty}\,\lesssim\,\log\Big(e\,+\,\mc N(t)\Big)\,.
\]
\end{lemma}

\begin{proof}
 The proof is very simple. The starting point is the logarithmic interpolation inequality of Proposition \ref{p:interpol} applied componentwise
to $\nabla u$.

First of all, we notice that
\[
 \left\|\nabla u\right\|_{B^\veps_{\infty,\infty}}\,\lesssim\,\mc N
\]
for $\veps>0$ small enough: this is a consequence of the fact that in the subcritical case, namely
for $(s,p,r)$ satisfying the strict inequality in \eqref{cond:Lipschitz},
one has the chain of continuous embeddings $B^s_{p,r}\,\hookrightarrow\, B^{s-d/p}_{\infty,r}\,\hookrightarrow\,B^{s-d/p}_{\infty,\infty}$,
with $s-d/p>1$.

Secondly, it is standard to bound
\[
 \left\|\nabla u\right\|_{B^0_{\infty,\infty}}\,\lesssim\,\left\|u\right\|_{L^2}\,+\,\left\|\o\right\|_{L^\infty}\,,
\]
where $\o\,=\,\d_1u^2-\d_2u^1$ is the vorticity of the fluid. Now, from \eqref{eq:eta-vort} and the estimates in
\eqref{est:eta_Leb}, \eqref{est:unif_D-rho} and \eqref{est:u-inf}, it follows that
\[
 \sup_{t\in[0,T[\,}\left\|\o(t)\right\|_{L^\infty}\,<\,+\infty\,.
\]
Therefore, using also \eqref{est:u-2} and the fact that the function $\beta(z)\,=\,z\,\log\left(e+\frac{c}{z}\right)$ is increasing on $\,]0,+\infty[\,$,
we deduce the claimed bound.
This concludes the proof of the lemma.
\end{proof}

The second statement establishes convenient bounds for the pressure gradient in terms of $\mc N$ and the Lipschitz norm of the velocity field.

\begin{lemma} \label{l:press-est}
Let the assumptions of Theorem \ref{th:cont-subcrit} be in force. Then the following bound holds true, for a suitable (implicit) multiplicative
constant $C>0$: for any fixed time $t\in[0,T[\,$, one has
\[
\left\|\nabla\Pi(t)\right\|_{L^2\cap B^s_{p,r}}\,\lesssim\,1\,+\,\Big(1\,+\,\left\|\nabla u(t)\right\|_{L^\infty}\Big)\,\mc N(t)\,.
\]
In addition, one also has, for any $t\in[0,T[\,$, the inequality
\[
\left\|\nabla\Pi(t)\right\|_{L^\infty}\,+\,\left\|\Delta\Pi(t)\right\|_{L^{p_0}}\,\lesssim\,1\,+\,\left\|\nabla u(t)\right\|_{L^\infty}\,.
\]
\end{lemma}

\begin{proof}
Essentially, we have to repeat the computations carried out in \cite{Brav-F} and notice, after performing suitable (small) changes,
that assumption \eqref{assump:weaker} is enough to get the sought inequalities.
For the sake of completeness, we present here most of the details.

We start by writing the elliptic equation satisfied by the pressure function:
\begin{equation} \label{eq:D-Pi}
 -\,\div\left(\frac{1}{\rho}\,\nabla\Pi\right)\,=\,\div(F)\,,\qquad\qquad \mbox{ with }\qquad 
F\,:=\,u\cdot\nabla u\,.
\end{equation}
It is obtained by dividing the momentum equation in \eqref{eq:dd-E} by $\rho$, which is possible owing to \eqref{est:rho-inf},
and then applying the divergence operator to the resulting expression.

Thanks to the Lax-Milgram theorem applied to \eqref{eq:D-Pi} (we refer to \tsl{e.g.} Lemma 2 in \cite{D1} for the details),
from \eqref{eq:D-Pi} one immediately infers that
\[ 
\frac{1}{\rho^*}\,\left\|\nabla\Pi\right\|_{L^2}\,\lesssim\,\left\|F\right\|_{L^2}\,\lesssim\,
\left\|u\right\|_{L^{q_0}}\,\left\|\nabla u\right\|_{L^{p_0}}\,, 
\] 
where $p_0$ is the exponent fixed in assumption \tbf{(A4)} and $q_0$ is defined by the relation $1/q_0 + 1/p_0 = 1/2$.
In particular, since $2<p_0< +\infty$, one has that $q_0\in\,]2,+\infty[\,$ as well; hence one can use estimates \eqref{est:u-2} and \eqref{est:u-inf} to
deduce that $\|u\|_{L^{q_0}}\lesssim C$. In addition, we apply Calder\'on-Zygmund theory and relation \eqref{eq:eta-vort} to get
\[
\left\|\nabla u\right\|_{L^{p_0}}\,\lesssim\,\left\|\o\right\|_{L^{p_0}}\,\lesssim\,\left\|\eta\right\|_{L^{p_0}}\,+\,\left\|u\cdot\nabla^\perp\rho\right\|_{L^{p_0}}\,
\lesssim\,\left\|\eta\right\|_{L^{p_0}}\,+\,\left\|u\right\|_{L^{p_0}}\,\left\|\nabla\rho\right\|_{L^\infty}\,.
\]
Thus, owing to the estimates established in \eqref{est:unif_D-rho} and \eqref{est:eta_Leb}, and using \eqref{est:u-2} and \eqref{est:u-inf} again,
we finally find that
\begin{align} \label{est:unif-Du_L^p}
 \sup_{t\in[0,T[\,}\left\|\nabla u(t)\right\|_{L^{p_0}}\,<\,+\,\infty\,,
\end{align}
which in turn implies that
\begin{equation} \label{est:Pi_L^2}
\sup_{t\in[0,T[\,}\left\|\nabla\Pi(t)\right\|_{L^2}\,\lesssim\,\sup_{t\in[0,T[\,}\left\|F(t)\right\|_{L^2}\,<\,+\infty\,.
\end{equation}

Next, we estimate the higher order regularity norm of $\nabla\Pi$, namely its Besov norm $B^s_{p,r}$. For this we use Lemma 3.12
of \cite{Brav-F}, which yields
 \begin{equation} \label{est:Pi-Besov-iniz}
\|\nabla \Pi \|_{B^{s}_{p,r}}\, \lesssim\,\left( 1 +  \| \nabla \rho  \|_{L^\infty}^{\theta} \right)\,\|F\|_{L^2}\,+\,
 \| \nabla \rho  \|_{B^{s-1}_{p,r}}\,\left\|\nabla\Pi\right\|_{L^\infty}\,+\, \left\| \rho\,  \div(F)  \right\|_{B^{s-1}_{p,r}}\,,
\end{equation}
for a suitable $\theta>1$. For the first term appearing in the right-hand side of the previous inequality, we can use the bounds established
in \eqref{est:unif_D-rho} and \eqref{est:Pi_L^2}. For the second term, instead, we use the interpolation inequality
\begin{equation} \label{est:Pi-interpol}
 \left\|\nabla\Pi\right\|_{L^\infty}\,\lesssim\,\left\|\nabla\Pi\right\|_{L^2}^\alpha\,\left\|\Delta\Pi\right\|_{L^{p_0}}^{1-\alpha}\,,
\end{equation}
for a suitable $\alpha\in\,]0,1[\,$, thanks to the fact that $p_0>d=2$. By developing the derivatives in \eqref{eq:D-Pi}, we find
an equation for $\Delta\Pi$:
\[
-\,\Delta\Pi\,=\,-\,\nabla\log\rho\cdot\nabla\Pi\,+\,\rho\,\nabla u:\nabla u\,. 
\]
From it, we easily deduce the following bound:
\begin{align*}
 \left\|\Delta\Pi\right\|_{L^{p_0}}\,&\lesssim\,\frac{1}{\rho_*}\,\left\|\nabla\rho\right\|_{L^{\infty}}\,\left\|\nabla\Pi\right\|_{L^{p_0}}\,+\,
 \rho^*\,\left\|\nabla u\right\|_{L^{p_0}}\,\left\|\nabla u\right\|_{L^\infty} \\
 &\lesssim\,\left\|\nabla\rho\right\|_{L^{\infty}}\,\left\|\nabla\Pi\right\|^{\alpha_1}_{L^2}\,\left\|\Delta\Pi\right\|^{1-\alpha_1}_{L^{p_0}}\,+\,
 \left\|\nabla u\right\|_{L^{p_0}}\,\left\|\nabla u\right\|_{L^\infty}\,,
\end{align*}
where, in passing from the first inequality to the second one, we have used an interpolation argument (for a suitable $\alpha_1\in\,]0,1[\,$) similar to
\eqref{est:Pi-interpol}. By a use of the Young inequality, together with \eqref{est:unif_D-rho}, \eqref{est:Pi_L^2} and \eqref{est:unif-Du_L^p},
we find, for any fixed time $t\in[0,T[\,$, the following estimate:
\begin{equation} \label{est:Pi-Lapl}
 \left\|\Delta\Pi(t)\right\|_{L^{p_0}}\,\lesssim\,1\,+\,\left\|\nabla u(t)\right\|_{L^\infty}\,.
\end{equation}
Plugging \eqref{est:Pi_L^2} and the just proven \eqref{est:Pi-Lapl} into \eqref{est:Pi-interpol}, in turn, implies an analogous bound also for the $L^\infty$
norm of the pressure gradient: for any $t\in[0,T[\,$, one has
\begin{equation} \label{est:Pi-inf}
 \left\|\nabla\Pi(t)\right\|_{L^\infty}\,\lesssim\,1\,+\,\left\|\nabla u(t)\right\|_{L^\infty}\,.
\end{equation}

Finally, we bound the $B^{s-1}_{p,r}$ norm of the term $\rho\,\div(F)$. For this, we use the trick of \cite{Brav-F} in order
to avoid the appearing of a bad $\left\|\nabla u\right\|_{L^\infty}^2$ term, and we write
\begin{align*}
\rho\,  \div(F)\,&=\,\rho\,\nabla u:\nabla u\,=\,\nabla\big(\rho\,u\big):\nabla u\,-\,\big(u\cdot\nabla u\big)\cdot\nabla\rho\,.
\end{align*}
By a repeated application of the tame estimates of Proposition \ref{p:tame} (recall that we are in the subcritical case), we then bound
\begin{align*}
\left\| \rho\,  \div(F)  \right\|_{B^{s-1}_{p,r}}\,&\lesssim\,\left\|\nabla u\right\|_{L^\infty}\,\left\|\nabla\big(\rho u\big)\right\|_{B^{s-1}_{p,r}}
\,+\,\left\|\nabla u\right\|_{B^{s-1}_{p,r}}\,\left\|\nabla\big(\rho u\big)\right\|_{L^\infty} \\
&\qquad\qquad \,+\,
\left\|u\cdot\nabla u\right\|_{B^{s-1}_{p,r}}\,\left\|\nabla\rho\right\|_{L^\infty}\,+\,
\left\|u\cdot\nabla u\right\|_{L^\infty}\,\left\|\nabla\rho\right\|_{B^{s-1}_{p,r}} \,.
\end{align*}
We now use the Leibniz rule and the Bony paraproduct decomposition to estimate
\begin{align*}
\left\|\nabla\big(\rho u\big)\right\|_{B^{s-1}_{p,r}}\,&\lesssim\,\left\|u\,\nabla\rho\right\|_{B^{s-1}_{p,r}}\,+\,\left\|\rho\,\nabla u\right\|_{B^{s-1}_{p,r}} \\
&\lesssim\,\left\|\nabla\rho\right\|_{B^{s-1}_{p,r}}\,\left\|u\right\|_{L^\infty}\,+\,\left\|u\right\|_{B^{s-1}_{p,r}}\,\left\|\nabla\rho\right\|_{L^\infty}\,+\,
\rho^*\,\left\|\nabla u\right\|_{B^{s-1}_{p,r}}\,, \\
\left\|u\cdot\nabla u\right\|_{B^{s-1}_{p,r}}\,&\lesssim\,
\|u\|_{B^{s-1}_{p,r}}\,\left\|\nabla u\right\|_{L^\infty}\,+\,\|u\|_{L^\infty}\,\left\|\nabla u\right\|_{B^{s-1}_{p,r}}\,.
\end{align*}
Thanks to the bounds established in \eqref{est:unif_D-rho} and \eqref{est:u-inf},
these inequalities yield the following estimate:
\begin{align*}
\left\| \rho\,  \div(F)  \right\|_{B^{s-1}_{p,r}}\,&\lesssim\,\left\|\nabla u\right\|_{L^\infty}\,\mc N\,+\,\mc N\,.
\end{align*}

In the end, inserting the previous bound and \eqref{est:Pi-inf} into \eqref{est:Pi-Besov-iniz}, we infer that
\[
 \|\nabla \Pi(t) \|_{B^{s}_{p,r}}\, \lesssim\,1\,+\,\left(1\,+\,\left\|\nabla u(t)\right\|_{L^\infty}\right)\,\mc N(t)
\]
for any $t\in[0,T[\,$, as claimed. The lemma is thus completely proven.
\end{proof}

In light of Lemmas \ref{l:Du-log-interp} and \ref{l:press-est}, it remains us to find a suitable bound for the quantity $\mc N(t)$.
We first need a preliminary estimate, which is the object of the next proposition.

\begin{prop} \label{p:N-bound}
Let the assumptions of Theorem \ref{th:cont-subcrit} be in force. Then, there exists an implicit constant $C>0$, depending on the fixed time $T$
and on the functional norms of the initial datum $\big(\rho_0,u_0\big)$, such that, for any fixed $t\in[0,T[\,$, one has the inequality
\[
 \mc N(t)\,\lesssim\,1\,+\,\int^t_0\Big(1\,+\,\left\|\nabla u(\t)\right\|_{L^\infty}\Big)\,\mc N(\t)\,\dd\t\,.
\]
\end{prop}

\begin{proof}
We borrow most of the arguments from \cite{Brav-F}, see the proof of Lemma 4.2 therein.

We first focus on the control of the Besov norm of the density.
Starting from the transport equation \eqref{eq:X} for $X=\nabla^\perp\rho$, a
standard argument (we refer \tsl{e.g.} to \cite{D1} or, in the case $p=+\infty$, to \cite{D:F}) yields, for any $t\in[0,T[\,$,
the bound
\begin{equation*}
\left\|\nabla \rho(t)\right\|_{B^{s-1}_{p,r}}\,\lesssim\,\left\|\nabla\rho_0\right\|_{B^{s-1}_{p,r}}\,+\,
\int^t_0\left(\|\nabla u \|_{L^{\infty}}\,\|\nabla \rho\|_{B^{s-1}_{p,r}}\, +\,\|\nabla \rho \|_{L^{\infty}}\, \|\nabla u\|_{B^{s-1}_{p,r}}\right)\,\dt\,,
\end{equation*}
which implies, thanks to \eqref{est:unif_D-rho}, the following inequality:
\[
\left\|\nabla \rho(t)\right\|_{B^{s-1}_{p,r}}\,\lesssim\,\left\|\nabla\rho_0\right\|_{B^{s-1}_{p,r}}\,+\,
\int^t_0\Big(1\,+\,\left\|\nabla u(\t)\right\|_{L^\infty}\Big)\,\mc N(\t)\,\dd\t\,.
\]

Next, we focus on the equation for $u$, which, thanks to the absence of vacuum regions set off by \eqref{est:rho-inf}, can be recasted in the following form:
\[ 
\d_tu\,+\,u\cdot\nabla u\,+\,\frac{1}{\rho}\,\nabla\Pi\,=\,0\,.
\] 
Then, similarly to what previously done for $\nabla\rho$, we can bound the norm of the velocity field, for any $t\in[0,T[\,$, as follows:
\begin{align*}
\|u(t)\|_{B^{s}_{p,r}}\,&\lesssim\,\|u_0\|_{B^{s}_{p,r}}\,+\,\int^t_0\left(\|\nabla u \|_{L^{\infty}}\,\|u\|_{B^{s}_{p,r}}\,+\,
\frac{1}{\rho_*}\,\left\|\nabla\Pi\right\|_{B^s_{p,r}}\,+\,\|\nabla\Pi\|_{L^\infty}\,\left\|\nabla\rho\right\|_{B^{s-1}_{p,r}}\right)\,\dd\t\,.
\end{align*}
Taking advantage of the estimates established in Lemma \ref{l:press-est} above, we then find
\[
\|u(t)\|_{B^{s}_{p,r}}\,\lesssim\,\|u_0\|_{B^{s}_{p,r}}\,+\,\int^t_0\left(\|\nabla u(\t)\|_{L^{\infty}}\,\mc N(\t)\,+\,
1\,+\,\Big(1\,+\,\left\|\nabla u(\t)\right\|_{L^\infty}\Big)\,\mc N(\t)\right)\,\dd\t\,.
\]

Gathering those estimates together, we finally deduce the sought inequality.
\end{proof}

Thanks to the previous preliminary results, we can now complete the proof of Theorem \ref{th:cont-subcrit}, hence also of the first claim
given in Theorem \ref{th:blow-up}.

\begin{proof}[Proof of Theorem \ref{th:cont-subcrit}]
As already pointed out, the goal is to find a uniform bound for the quantity $\mc N$ on the time interval $[0,T[\,$.
For this, we plug the logarithmic inequality of Lemma \ref{l:Du-log-interp} into the bound provided by Proposition \ref{p:N-bound}: we find
\[
 \mc N(t)\,\leq\,C_1\,+\,C_2\int^t_0\mc N(\t)\,\log\Big(e\,+\,\mc N(\t)\Big)\,\dd\t
\]
for any $t\in[0,T[\,$, for two suitable constants $C_1>0$ and $C_2>0$, possibly depending on the time $T$ and on the functional norms of the initial density
and velocity field. By an application of the Osgood lemma (see \tsl{e.g.} Lemma 3.4 of \cite{BCD}), we find
\[
\log\log\Big(e\,+\,\mc N(t)\Big)\,\leq\,\log\log(e\,+\,C_1)\,+\,C_2\,t\,,
\]
which in turn implies that
\[
 \sup_{t\in[0,T[\,}\mc N(t)\,<\,+\,\infty\,.
\]
Thanks to this property and Lemma \ref{l:press-est}, we see that one also has
\[
 \sup_{t\in[0,T[\,}\left\|\nabla\Pi(t)\right\|_{L^2\cap B^s_{p,r}}\,<\,+\,\infty\,.
\]

In the end, we have proved that the quantity
\[
 E(t)\,:=\,\left\|\rho\right\|_{L^\infty}\,+\,\left\|\nabla\rho(t)\right\|_{B^{s-1}_{p,r}}\,+\,\left\|u(t)\right\|_{L^2\cap B^s_{p,r}}\,+\,
\left\|\nabla\Pi(t)\right\|_{L^2\cap B^s_{p,r}}\,,
\]
remains bounded on the time interval $[0,T[\,$. Thus, a standard argument applies to extend the solution $\big(\rho,u,\nabla\Pi\big)$ beyond the time $T$
while keeping the same regularity properties. This completes the proof of Theorem \ref{th:cont-subcrit}.
\end{proof}

We conclude this part with a remark concerning the regularity of the pressure.
\begin{rmk} \label{r:pressure}
By the momentum equation in \eqref{eq:dd-E}, the property $\d_X\rho=0$ and the commutator relation \eqref{eq:comm}, one deduces an equation
for $\d_Xu$: one has
\[
 \rho\,\d_t\d_Xu\,+\,\rho\,u\cdot\nabla\d_Xu\,=\,-\,\d_X\nabla\Pi\,.
\]
Owing to this and classical $L^\infty$ estimates for transport equations, it would be easy to get a continuation condition in terms of the geometric
regularity of the pressure gradient, at least in the subcritical case (the critical case looks more involved and goes beyond the scopes
of the present paper). For instance, condition \eqref{eq:cont_subcrit} can be replaced by
\[
 \int^T_0\left\|\d_{X(t)}\nabla\Pi(t)\right\|_{L^\infty}\,\dd t\,<\,+\,\infty\,.
\]
\end{rmk}

\section{The critical case: proof of Theorems \ref{th:cont-crit_sum} and \ref{th:cont-crit}} \label{s:proof}
Thanks to the uniform estimates established in Section \ref{s:prelim} and the material collected in Section \ref{s:tools}, we are in the position
of proving Theorems \ref{th:cont-crit_sum} and \ref{th:cont-crit}. This will yield also the second part of Theorem \ref{th:blow-up}, thus
completing its proof.

Since condition \eqref{eq:cont-cond} is stronger than \eqref{eq:cont-cond_sum}, we will limit ourselves to prove Theorem \ref{th:cont-crit}. Actually,
condition \eqref{eq:cont-cond_sum} will arise naturally in the course of our proof; so, at the right point of our argument,
we will explain how to get Theorem \ref{th:cont-crit_sum} directly (see Remark \ref{r:proof_sum} in this respect).

\medbreak
To begin with, we recall that, thanks to Proposition \ref{p:cont:pri}, under our assumptions the solution $\big(\rho,u,\nabla\Pi\big)$
to system \eqref{eq:dd-E} can be continued beyond a time $T>0$ as soon as 
\begin{equation} \label{eq:cont-condit-Du}
 \int^T_0\left\|\nabla u(t)\right\|_{L^\infty}\,\dt\,<\,+\,\infty\,.
\end{equation}
Thus, the proof of Theorem \ref{th:cont-crit} (hence, the proof of the second part of Theorem \ref{th:blow-up}) reduces to the proof of the next proposition.

\begin{prop} \label{p:Du-bound}
Let the hypotheses of Theorem \ref{th:cont-crit} be in force. Assume that the time $T>0$ is such that condition \eqref{eq:cont-cond} holds true.

Then, also condition \eqref{eq:cont-condit-Du} holds true for the same time $T$.
\end{prop}

The rest of this section is devoted to the proof of Proposition \ref{p:Du-bound}. It relies on another use of formula \eqref{eq:Helm-mom}, which allows us to
find a convenient expression for $\nabla u$.

\subsection{Reduction to estimating singular integrals} \label{ss:reduction}

As a first step, we use the definition of the momentum $m$ to write, for any $j\in\{1,2\}$, that $\d_j u\,=\,\d_j\big(m/\rho\big)$.
By expanding the derivatives on the right-hand side\footnote{Here and throughout this work, we use the convention that
$Du$ denotes the Jacobian matrix of $u$ and $\nabla u =\,^tDu$ its transpose matrix.}, we get the formula
\begin{align}
\label{eq:Du}
\nabla u\,&=\,-\,\frac{1}{\rho^2}\,\nabla\rho\otimes m 
\,+\,\frac{1}{\rho}\,
\nabla m\,.
\end{align}

Thanks to the uniform bounds \eqref{est:rho-inf}, \eqref{est:unif_D-rho} and \eqref{est:u-inf}, one can easily control the $L^\infty$ norm
of the first term in the right-hand side of the previous relation: 
\begin{align}
\label{est:Du-1}
\sup_{t\in[0,T[\,}\left\|\frac{1}{\rho^2(t)}\,\nabla\rho(t)\otimes m(t)\right\|_{L^\infty}\,<\,+\,\infty\,.
\end{align}

On the other hand, for the second term appearing in \eqref{eq:Du}, using \eqref{eq:Helm-mom} we can write
\[
\left\|\frac{1}{\rho}\,\nabla m\right\|_{L^\infty}\,\lesssim\,\left\|\nabla m\right\|_{L^\infty}\,\lesssim\,S_1\,+\,S_2\,,
\]
where we have defined
\[
 S_1\,:=\,\left\|\nabla\nabla^\perp(-\Delta)^{-1}\eta\right\|_{L^\infty}\qquad \mbox{ and }\qquad
 S_2\,:=\,\left\|\nabla^2(-\Delta)^{-1}\big(u\cdot\nabla\rho\big)\right\|_{L^\infty}\,.
\]
The difficulty then comes from the fact that we need to bound in $L^\infty$ two singular integral operators applied to $L^\infty$ functions, but,
as is well-known, the singular integral operators are not bounded operators over $L^\infty$.

How to control the $L^\infty$ norm of singular integral operators analogous to the ones under study is by now quite well-understood. There are essentially
two possibilities.
The first one is to increase the regularity and work in the class of H\"older spaces; however, we want to avoid this approach here, as we aim to work
with minimal regularity assumptions.
Another possibility, somehow minimal in terms of smoothness, is inspired by the pioneering works \cite{Ch_1991, Ch_1993}
by Chemin and consists in working with \emph{striated} regularity conditions.
However, as explained in Remark \ref{r:striated}, this strategy would also lead to formulate additional hypotheses, not needed here.

After this discussion, we resume with our argument: in the next subsection, we are going to bound $S_2$ and $S_1$ respectively, by imposing
minimal regularity requirements. This naturally makes the Besov regularity class $B^0_{\infty,1}$ appear in our study.
Specifically, we are led to bound the norm of some suitable quantities in that space: we will deal with this issue in Subsection \ref{ss:B^0}.
Finally, in Subsection \ref{ss:end} we will put all the estimates together and prove Proposition \ref{p:Du-bound}.

\subsection{Bounds for the singular integrals} \label{ss:bound-S}
In this subsection, we show how to estimate the two singular integral terms $S_1$ and $S_2$ by making minimal regularity requirements appear.

Let us first deal with the term
\[
  S_2\,:=\,\left\|\nabla^2(-\Delta)^{-1}\big(u\cdot\nabla\rho\big)\right\|_{L^\infty}\,.
\]

We notice that the singularity of the singular integral operator disappears when considered at high frequencies, whereas, for low frequencies, one disposes
of the Bernstein inequalities to reduce the $L^\infty$ bound to a $L^p$ bound, for $p<+\infty$. Therefore, we write the decomposition
\[
 \nabla^2(-\Delta)^{-1}\big(u\cdot\nabla\rho\big)\,=\,\Delta_{-1}\nabla^2(-\Delta)^{-1}\big(u\cdot\nabla\rho\big)\,+\,
\big(\Id-\Delta_{-1}\big)\nabla^2(-\Delta)^{-1}\big(u\cdot\nabla\rho\big)\,.
\]
Using that $u\cdot\nabla\rho = \div\big(\rho\,u\big) = \div m$ and the first Bernstein inequality, we then get that
\begin{align*}
 S_2\,&\lesssim\,\left\|\Delta_{-1}\nabla^2(-\Delta)^{-1}\div m\right\|_{L^2}\,+\,
\left\|\big(\Id-\Delta_{-1}\big)\nabla^2(-\Delta)^{-1}\big(u\cdot\nabla\rho\big)\right\|_{L^\infty} \\
&\lesssim\,\left\|m\right\|_{L^2}\,+\,\left\|\big(\Id-\Delta_{-1}\big)\nabla^2(-\Delta)^{-1}\big(u\cdot\nabla\rho\big)\right\|_{L^\infty} \\
&\lesssim\,1\,+\,\left\|\big(\Id-\Delta_{-1}\big)\nabla^2(-\Delta)^{-1}\big(u\cdot\nabla\rho\big)\right\|_{L^\infty}\,,
\end{align*}
where, in passing from the first to the second line, we have made use of Plancherel's theorem,
while the passage from the second to the third line relies on the bounds \eqref{est:rho-inf} and \eqref{est:u-2}.

Now that we have cut the low frequencies out, the problem consists in bounding the second term on the right-hand side of the last inequality. The natural way to
do that is to make the stronger $B^0_{\infty,1}$ norm come into play: as, for any $j\geq0$, the $0$-th order operator
$\Delta_j\nabla^2(-\Delta)^{-1}$ is smooth and supported on an annulus, Lemma 2.2 of \cite{BCD} yields
\begin{align*}
 \left\|\big(\Id-\Delta_{-1}\big)\nabla^2(-\Delta)^{-1}\big(u\cdot\nabla\rho\big)\right\|_{L^\infty}
&\lesssim\,\sum_{j\geq0}\left\|\Delta_j\nabla^2(-\Delta)^{-1}\big(u\cdot\nabla\rho\big)\right\|_{L^\infty} \\
&\lesssim\,\sum_{j\geq0}\left\|\Delta_j\big(u\cdot\nabla\rho\big)\right\|_{L^\infty} \\
&\lesssim\,\left\|u\cdot\nabla\rho\right\|_{B^0_{\infty,1}}\,.
\end{align*}
In this way, we have avoided the singularity of the operator $\nabla^2(-\Delta)^{-1}$, but we have now the problem of estimating the higher order
$B^0_{\infty,1}$ norm. In addition, notice that we have to estimate a product in that space, which however is \emph{not} an algebra.

However, thanks to the divergence-free property of $u$, we can follow \cite{D:F} (see also \cite{Brav-F}) and write the following Bony paraproduct
decomposition of $u\cdot\nabla\rho$:
\begin{equation} \label{decomp:paraprod}
 u\cdot\nabla\rho\,=\,\mc T_u\cdot\nabla\rho\,+\,\mc T_{\nabla\rho}\cdot u \,+\,\sum_{k=1,2}\d_k\mc R(\rho,u^k)\,.
\end{equation}
From the paraproduct estimates of Proposition \ref{p:op}, we then obtain the following bounds:
\begin{align*}
\left\|u\cdot\nabla\rho\right\|_{B^0_{\infty,1}}\,&\lesssim\,\left\|\nabla\rho\right\|_{B^0_{\infty,1}}\,\|u\|_{L^\infty}\,+\,
\left\|u\right\|_{B^0_{\infty,1}}\,\left\|\nabla\rho\right\|_{L^\infty}\,+\,\left\|\mc R(\rho,u)\right\|_{B^1_{\infty,1}} \\
&\lesssim\,\left\|\nabla\rho\right\|_{B^0_{\infty,1}}\,\|u\|_{L^\infty}\,+\,
\left\|u\right\|_{B^0_{\infty,1}}\,\left\|\nabla\rho\right\|_{L^\infty}\,+\,\left\|u\right\|_{B^0_{\infty,1}}\,\left\|\rho\right\|_{B^1_{\infty,\infty}}\,.
\end{align*}
Making use of the embedding $W^{1,\infty}\hookrightarrow B^1_{\infty,\infty}$ and
of the uniform estimates \eqref{est:rho-inf}, \eqref{est:unif_D-rho} and \eqref{est:u-inf}, in turn we get
\begin{align*}
\left\|u\cdot\nabla\rho\right\|_{B^0_{\infty,1}}\,&\lesssim\,\left\|\nabla\rho\right\|_{B^0_{\infty,1}}\,+\,\left\|u\right\|_{B^0_{\infty,1}}\,.
\end{align*}

In light of the previous computations, in the end we get the following bound for $S_2$:
\begin{equation} \label{est:S_2}
S_2\,\lesssim\,1\,+\,\left\|\nabla\rho\right\|_{B^0_{\infty,1}}\,+\,\left\|u\right\|_{B^0_{\infty,1}}\,.
\end{equation}

\medbreak
We now turn our attention to the estimate of the term $S_1$, defined as
\[
 S_1\,:=\,\left\|\nabla\nabla^\perp(-\Delta)^{-1}\eta\right\|_{L^\infty}\,.
\]
Repeating \tsl{mutatis mutandis} the argument employed at the beginning of this subsection, based on cutting low and high frequencies
and using Bernstein's inequalities, we get
\begin{align}
  \label{est:S_1}
S_1\,&\lesssim\,\left\|\Delta_{-1}\nabla^2(-\Delta)^{-1}\curl m\right\|_{L^\infty}\,+\,
\left\|\big(\Id-\Delta_{-1}\big)\nabla^2(-\Delta)^{-1}\eta\right\|_{L^\infty} \\
\nonumber
&\lesssim\,\left\|m\right\|_{L^2}\,+\,\left\|\big(\Id-\Delta_{-1}\big)\nabla^2(-\Delta)^{-1}\eta\right\|_{L^\infty} \\
\nonumber
&\lesssim\,1\,+\,\left\|\eta\right\|_{B^{0}_{\infty,1}}\,.
\end{align}

\subsection{Estimating the Besov norms of regularity index $0$} \label{ss:B^0}
Gathering inequalities \eqref{est:Du-1}, \eqref{est:S_2} and \eqref{est:S_1} together, from relation \eqref{eq:Du} we infer that,
for any given time $t\in[0,T[\,$, one has
\begin{equation*}
\left\|\nabla u(t)\right\|_{L^\infty}\,\lesssim\,1\,+\,\left\|\nabla\rho(t)\right\|_{B^0_{\infty,1}}\,+\,\left\|u(t)\right\|_{B^0_{\infty,1}}
\,+\,\left\|\eta(t)\right\|_{B^{0}_{\infty,1}}\,.
\end{equation*}
Applying the interpolation inequality \eqref{est:interp-u} with $X=B^0_{\infty,1}$, we then get, for any time $t\in[0,T[\,$ fixed, the bound
\begin{equation} \label{est:Du_with_B}
\left\|\nabla u(t)\right\|_{L^\infty}\,\lesssim\,1\,+\,\left\|\nabla\rho(t)\right\|_{B^0_{\infty,1}}\,+\,\left\|\eta(t)\right\|_{B^{0}_{\infty,1}}\,. 
\end{equation}
The next goal is to bound the $B^0_{\infty,1}$ norms appearing on the right-hand side of the previous inequality.
To do that, we will make use of Theorem \ref{th:B^0}.

Before going on, in order to simplfy our argument, we introduce the following convenient notation: for $t\in[0,T[\,$, we set
\begin{equation} \label{def:U(t)}
 U(t)\,:=\,1\,+\,\int^t_0\left\|\nabla u(\t)\right\|_{L^\infty}\,\dd\t\,.
\end{equation}

\paragraph*{Use of the improved transport estimates.}
Consider the term $\nabla\rho$ first and observe that, in order to bound it, it is enough to control the $B^0_{\infty,1}$ of the vector
field $X=\nabla^\perp\rho$. Thus, applying the improved transport estimates of Theorem \ref{th:B^0} to equation \eqref{eq:X}, we find
\begin{equation} \label{est:Drho-B^0}
\left\|\nabla\rho(t)\right\|_{B^0_{\infty,1}}\,\lesssim\,U(t)\,\left(\left\|\nabla\rho_0\right\|_{B^0_{\infty,1}}\,+\,
\int^t_0\left\|\d_Xu(\t)\right\|_{B^{0}_{\infty,1}}\,\dd\t\right)
\end{equation}
for all time $t\in[0,T[\,$. Analogously, we also get, for any $t\in[0,T[\,$, the estimate
\begin{equation} \label{est:eta_Besov}
\left\|\eta(t)\right\|_{B^0_{\infty,1}}\,\lesssim\,U(t)\,\left(\left\|\eta_0\right\|_{B^0_{\infty,1}}\,+\,
\int^t_0\left\|\d_Xu(\t)\cdot u(\t)\right\|_{B^{0}_{\infty,1}}\,\dd\t\right)\,.
\end{equation}



At this point, we observe that, thanks to the above estimate, one is already in the position of proving Theorem \ref{th:cont-crit_sum}.
Let us explain this in the next remark.
\begin{rmk} \label{r:proof_sum}
We can insert the bounds given in \eqref{est:Drho-B^0} and \eqref{est:eta_Besov} into estimate \eqref{est:Du_with_B}.
Observing that $\d_Xu\cdot u\,=\,\d_X|u|^2/2$, we then use assumption \eqref{eq:cont-cond_sum} to get the following differential inequality:
\[
 U'(t)\,\leq\,C_1\,+\,C_2\,U(t)\,,
\]
for two suitable positive constants $C_1$ and $C_2$. With this inequality at hand, it is easy to deduce, by the use of a Gr\"onwall argument, that
$U(t)\,\leq\,C$ for all $t\in[0,T[\,$. This leads to the same conclusion of Proposition \ref{p:Du-bound} under condition \eqref{eq:cont-cond_sum},
hence it also yields the proof of Theorem \ref{th:cont-crit_sum}.
\end{rmk}

This having been pointed out, let us resume with the proof of Proposition \ref{p:Du-bound} under assumption \eqref{eq:cont-cond}.
All the difficulty is hidden in the control of the norm of the non-linear term $\d_Xu\cdot u$ inside the integral:
we want to obtain such control by requiring minimal regularity
assumptions. Indeed, on the one hand the space $B^0_{\infty,1}$ is not an algebra and, on the other hand, when estimating that term,
we must avoid the appearing of any power of the function $U'(t)\,=\,\|\nabla u(t)\|_{L^\infty}$, which would make
our result impossible to prove.


\paragraph*{Bouns for the non-linear term $\d_Xu\cdot u$.}
In order to treat the non-linear term appearing in \eqref{est:eta_Besov}, we proceed by writing its Bony paraproduct decomposition:
\[
 \d_Xu\cdot u\,=\,\mc T_{\d_Xu}\cdot u\,+\,\mc T_u\cdot\d_Xu\,+\,\mc R(\d_Xu,u)\,.
\]
From the classical paraproduct estimates of Proposition \ref{p:op}, we see that
\[
  \left\|\mc T_{\d_Xu}\cdot u\right\|_{B^0_{\infty,1}}\,\lesssim\,\left\|\d_Xu\right\|_{L^\infty}\,\left\|u\right\|_{B^0_{\infty,1}}\qquad \mbox{ and }\qquad
\left\|\mc T_{u}\cdot \d_Xu\right\|_{B^0_{\infty,1}}\,\lesssim\,\left\|u\right\|_{L^\infty}\,\left\|\d_Xu\right\|_{B^0_{\infty,1}}\,.
\]
Now, we can control the 
$B^0_{\infty,1}$ norm of $u$ by Proposition \ref{p:interpol} and a subsequent
application of inequality \eqref{est:interp-u}. We then find the following estimates:
\begin{align}
 \label{est:T-1}
 \left\|\mc T_{\d_Xu}\cdot u\right\|_{B^0_{\infty,1}}\,&\lesssim\,\left\|\d_Xu\right\|_{L^\infty}\,\left\|u\right\|_{B^0_{\infty,\infty}}\,
\left(1\,+\,\log\left(1\,+\,\frac{\left\|\nabla u\right\|_{L^\infty}}{\|u\|_{B^0_{\infty,\infty}}}\right)\right)\,, \\
\label{est:T-2}
\left\|\mc T_{u}\cdot \d_Xu\right\|_{B^0_{\infty,1}}\,&\lesssim\,\left\|\d_Xu\right\|_{B^0_{\infty,1}}\,\left\|u\right\|_{L^\infty}\,.
\end{align}
All the difficulty of working with critical regularity norms appears when dealing with the remainder term. For estimating it,
we cannot rely on Proposition \ref{p:op}, which would yield a bound on the weaker $B^0_{\infty,\infty}$ norm only.
Several strategies would be possible, but they would lead to worse results. Let us spend some time to discuss them in the next remark.

\begin{rmk} \label{r:power}
In order to bound $\mc R(\d_Xu,u)$, the first attemp could consist in working in a positive regularity framework and estimating it, for any $\veps>0$ small,
in the following way,
 \[
\left\|\mc R(\d_Xu,u)\right\|_{B^0_{\infty,1}}\,\lesssim\,\left\|\mc R(\d_Xu,u)\right\|_{B^{\veps}_{\infty,\infty}}\,\lesssim\,
\left\|\d_Xu\right\|_{L^\infty}\,\left\|u\right\|_{B^{\veps}_{\infty,1}}\,,
\]
and, then, use the interpolation inequality \eqref{est:interp-u}. However, this would have dramatic consequences:
indeed, this quantity is already multiplied by a factor $U(t)$, keep in mind \eqref{est:eta_Besov}, therefore it would cause
the presence of a factor $U^2(t)$ in the final estimate. Such a quadratic growth, of course, would completely
destroy the argument.

For the same reason, the equality $\d_Xu\cdot u = \frac{1}{2}\,\d_X|u|^2$ is out of use here. 
First of all, it does not allow to ``factorise'' one derivative outside (as it was the case in \eqref{decomp:paraprod} above, for instance),
as the above equality needs symmetry between the two factors, which is broken when writing $\mc R(\d_Xu,u)$. 
Secondly, even using that equality (for instance, before the paraproduct decomposition), one would find a dangerous factor $\|u\|_{B^1_{\infty,1}}^2$,
which, for precisely the same reason as above, cannot be admitted in the estimates.

Finally, directly using Proposition \ref{p:interpol} on the remainder term $\mc R(\d_Xu,u)$ would lead to the presence of a bothering
factor $\left\|\d_Xu\right\|_{B^0_{\infty,1}}$ in the denominator of the logarithm, a factor that we want to avoid here.
\end{rmk}

The estimate of $\mc R(\d_Xu,u)$ in $B^0_{\infty,1}$ will rely, instead, on the use of Lemma \ref{l:interp-log} and Proposition \ref{p:log-interp}.
More precisely, we first employ Lemma \ref{l:interp-log} with some $\alpha>1$, which can be chosen arbitrarily close to $1$, to bound
\begin{equation} \label{est:Rem_provvis}
 \left\|\mc R(\d_Xu,u)\right\|_{B^0_{\infty,1}}\,\lesssim\,\left\|\mc R(\d_Xu,u)\right\|_{B^{0}_{\infty,\infty}}^{1-1/\alpha}
\left\|\mc R(\d_Xu,u)\right\|_{B^{\alpha\log}_{\infty,\infty}}^{1/\alpha}\,\,.
\end{equation}

We focus on the $B^0_{\infty,\infty}$ factor first. For bounding it, we can directly use Proposition \ref{p:op}: we obtain
\begin{equation} \label{est:R_B^0}
\left\|\mc R(\d_Xu,u)\right\|_{B^{0}_{\infty,\infty}}\,\lesssim\,\left\|\d_Xu\right\|_{B^{0}_{\infty,1}}\,\left\|u\right\|_{B^0_{\infty,\infty}}\,.
\end{equation}

Next, we consider the factor in \eqref{est:Rem_provvis} presenting the $B^{\alpha\log}_{\infty,\infty}$ norm. We first use
Theorem \ref{t:log-r} and then (since $\alpha>1$ can be chosen as close to $1$ as needed) Proposition \ref{p:log-interp}: we find the estimate
\begin{align*}
\left\|\mc R(\d_Xu,u)\right\|_{B^{\alpha\log}_{\infty,\infty}}\,&\lesssim\,\left\|u\right\|_{B^{\alpha\log}_{\infty,\infty}}\,\left\|\d_Xu\right\|_{B^0_{\infty,1}} \\
&\lesssim\,\left\|u\right\|_{B^0_{\infty,\infty}}\,
\left(1\,+\,\log\left(1\,+\,\frac{\left\|\nabla u\right\|_{L^\infty}}{\|u\|_{B^0_{\infty,\infty}}}\right)\right)^\alpha\,\left\|\d_Xu\right\|_{B^0_{\infty,1}}\,,
\end{align*}
where we have also used the continuous embedding $L^\infty\hookrightarrow B^0_{\infty,\infty}$.

Combining this inequality with \eqref{est:R_B^0}, from \eqref{est:Rem_provvis} we finally find
\begin{align}
\label{est:Rem_final}
 \left\|\mc R(\d_Xu,u)\right\|_{B^0_{\infty,1}}\,\lesssim\,\left\|\d_Xu\right\|_{B^0_{\infty,1}}\,\left\|u\right\|_{B^0_{\infty,\infty}}\,
\left(1\,+\,\log\left(1\,+\,\frac{\left\|\nabla u\right\|_{L^\infty}}{\|u\|_{B^0_{\infty,\infty}}}\right)\right)\,.
\end{align}

We can now gather estimates \eqref{est:T-1}, \eqref{est:T-2} and \eqref{est:Rem_final} together. 
Using again the continuous embedding $L^\infty\hookrightarrow B^0_{\infty,\infty}$ and the fact that 
the function $\beta(z)\,=\,z\,\log\left(e+\frac{c}{z}\right)$ is monotonically
increasing for $z>0$, we deduce
\[
\left\|\d_Xu\cdot u\right\|_{B^0_{\infty,1}}\,\lesssim\,\left\|\d_Xu\right\|_{B^0_{\infty,1}}\,
\left\|u\right\|_{L^\infty}\,
\left(1\,+\,\log\left(1\,+\,\frac{\left\|\nabla u\right\|_{L^\infty}}{\|u\|_{L^\infty}}\right)\right)\,.
\]
Since one has $1+\log(1+z)\leq 2\log(e+z)$ for $z\geq0$, plugging the previous bound into inequality \eqref{est:eta_Besov} yields
\begin{align*}
\left\|\eta(t)\right\|_{B^0_{\infty,1}}\,&\lesssim\,U(t)\,\Bigg(\left\|\eta_0\right\|_{B^0_{\infty,1}}\,+\,
\int^t_0\left\|\d_Xu\right\|_{B^0_{\infty,1}}\,
\left\|u\right\|_{L^\infty}\,\log\left(e\,+\,\frac{U'(\t)}{\|u\|_{L^\infty}}\right)\,\dd\t\Bigg)\,.
\end{align*}
Finally, thanks to bound \eqref{est:u-inf} and recalling that $\beta(z)\,=\,z\,\log\left(e+\frac{c}{z}\right)$ is increasing for $z>0$, we find,
for any time $t\in[0,T[\,$, the estimate
\begin{align}
\label{est:eta_Besov_fin}
\left\|\eta(t)\right\|_{B^0_{\infty,1}}\,&\lesssim\,U(t)\,\Bigg(\left\|\eta_0\right\|_{B^0_{\infty,1}}\,+\,
\int^t_0\left\|\d_Xu\right\|_{B^{0}_{\infty,1}}\,\log\left(e\,+\,U'(\t)\right)\,\dd\t\Bigg)\,.
\end{align}

\subsection{End of the argument} \label{ss:end}
We are now ready to conclude the proof of Proposition \ref{p:Du-bound}. We insert inequalities \eqref{est:Drho-B^0} and \eqref{est:eta_Besov_fin}
into \eqref{est:Du_with_B}: by making use of the pointwise assumption \eqref{eq:cont-cond}, we find
\begin{align} \label{est:U'-provvis}
U'(t)\,\lesssim\,1\,+\,U(t)\,+\,U(t)\,\int^t_0\log\left(e\,+\,U'(\t)\right)\,\dd\t
\end{align}
for any time $t\in[0,T[\,$. Notice that this is the only point where we need the pointwise assumption in \eqref{eq:cont-cond}.
Applying the Jensen inequality to the concave function $\Phi(z)\,:=\,\log\big(e+z\big)$,
we can write
\[
\frac{1}{t}\,\int^t_0\log\left(e\,+\,U'(\t)\right)\,\dd\t\,\leq\,\log\left(e\,+\,\frac{1}{t}\,U(t)\right)\,.
\]
Thanks to this, from \eqref{est:U'-provvis} we infer
\begin{align*}
U'(t)\,&\lesssim\,1\,+\,U(t)\,+\,U(t)\,t\,\log\left(e\,+\,\frac{1}{t}\,U(t)\right) \\
&\lesssim\,1\,+\,U(t)\,+\,U(t)\,T\,\log\left(e\,+\,\frac{1}{T}\,U(t)\right)\,,
\end{align*}
where, in passing from the first to the second inequality, we have used once again the fact that $\beta(z)\,=\,z\,\log\left(e+\frac{c}{z}\right)$ is increasing
on $\,]0,+\infty[\,$.

We observe that the factor $T$ in the previous bound can be treated as an additional multiplicative constant, as this time is fixed in our argument.
Therefore, in the end we have obtained that the function $U(t)$, defined in \eqref{def:U(t)}, satisfies the differential inequality
\[
 U'(t)\,\leq\,C_1\,+\,C_2\,U(t)\,\log\left(e\,+\,U(t)\right)\,,
\]
for suitable positive constants $C_1$ and $C_2$. From this, it is easy to see that the function $U$ must be uniformly bounded for all $t\in[0,T[\,$.
Hence, the integral condition \eqref{eq:cont-condit-Du} is verified, which implies that Proposition \ref{p:Du-bound} is finally proven.
As already remarked, this immediately yields the conclusion of Theorem \ref{th:cont-crit}, which is thus proven as well.



\addcontentsline{toc}{section}{References}
{\small

\begin{thebibliography}{xxx}

\bibitem{Korean} H. Bae, W. Lee, J. Shin:
{\it A blow-up criterion for the inhomogeneous incompressible Euler equations}.
Nonlinear Anal., {\bf 196} (2020), 111774, 9 pp.

\bibitem{BCD} H. Bahouri, J.-Y. Chemin, R. Danchin:
{\it ``Fourier analysis and nonlinear partial differential equations''}.
Grundlehren der mathematischen Wissenschaften [Fundamental Principles of Mathematical Sciences], 343. Springer, Heidelberg, 2011.

\bibitem{BKM} J. T. Beale, T. Kato, A. J. Majda:
{\it Remarks on the breakdown of smooth solutions for the 3-D Euler equations}.
Comm. Math. Phys., {\bf 94} (1984), n. 1, 61-66.

\bibitem{Bony} J.-M. Bony: 
{\it Calcul symbolique et propagation des singularit\'es pour les \'equations aux d\'eriv\'ees partielles non lin\'eaires}.
Ann. Sci. \'Ecole Norm. Sup., {\bf 14} (1981), n. 2, 209-246.

\bibitem{Brav-F} M. Bravin, F. Fanelli:
{\it Global existence for non-homogeneous incompressible inviscid fluids in presence of Ekman pumping}.
Commun. Contemp. Math., accepted for publication (2025).

\bibitem{Ch_1991} J.-Y. Chemin:
{\it Sur le mouvement des particules d'un fluide parfait incompressible bidimensionel}.
Invent. Math., {\bf 103} (1991), n. 3, 599-629.

\bibitem{Ch_1993} J.-Y. Chemin:
{\it Persistance des structures g\'eom\'etriques li\'ees aux poches de tourbillon}.
Ann. Sci. \'Ec. Norm. Sup\'er., {\bf 26} (1993), n. 4, 1-26.

\bibitem{Ch_1995} J.-Y. Chemin:
{\it ``Fluides parfaits incompressibles''}.
Ast\'erisque, {\bf 230}, 1995.

\bibitem{CWZZ} Q. Chen, D. Wei, P. Zhang, Z. Zhang:
{\it Nonlinear inviscid damping for $2$-D inhomogeneous incompressible Euler equations}.
Submitted (2023), arxiv preprint \texttt{2303.14858v1}.

\bibitem{C-DS-F-M} F. Colombini, D. Del Santo, F. Fanelli, G. M\'etivier:
{\it The well-posedness issue in Sobolev spaces for hyperbolic systems with Zygmund-type coefficients.}.
Comm. Partial Differential Equations, {\bf 40} (2015), n. 11, 2082-2121.

\bibitem{C-M} F. Colombini, G. M\'etivier,
{\it The Cauchy problem for wave equations with non-Lipschitz coefficients; application to continuation of solutions of
some nonlinear wave equations}.
Ann. Sci. \'Ec. Norm. Sup\'er. (4), {\bf 41} (2008), n. 2, 177-220.

\bibitem{C-Feff} P. Constantin, C. Fefferman:
{\it Direction of vorticity and the problem of global regularity for the Navier-Stokes equations}.
Indiana Univ. Math. J., {\bf 42} (1993), n. 3, 775-789.

\bibitem{C-Feff-M} P. Constantin, C. Fefferman, A. J. Majda:
{\it Geometric constraints on potentially singular solutions for the 3-D Euler equations}.
Comm. Partial Differential Equations, {\bf 21} (1996), n. 3-4, 559-571.

\bibitem{D_2006} R. Danchin:
{\it The inviscid limit for density-dependent incompressible fluids}.
Ann. Fac. Sci. Toulouse Math. (6), {\bf 15} (2006), n. 4, 637-688.

\bibitem{D1} R. Danchin:
{\it On the well-posedness of the incompressible density-dependent Euler equations in the $L^p$ framework}.
J. Differential Equations, {\bf 248} (2010), n. 8, 2130-2170.

\bibitem{D:F} R. Danchin, F. Fanelli:
{\it The well-posedness issue for the density-dependent Euler equations in endpoint Besov spaces}.
J. Math. Pures Appl. (9), {\bf 96} (2011), n. 3, 253-278.

\bibitem{F_thesis} F. Fanelli:
{\it ``Mathematical analysis of models of non-homogeneous fluids and of hyperbolic operators with low-regularity coefficients''}.
Ph.D. thesis, Scuola Internazionale Superiore di Studi Avanzati (Trieste, Italy) and Universit\'e Paris-Est (Cr\'eteil, France), 2012.  

\bibitem{F_2012} F. Fanelli:
{\it Conservation of geometric structures for non-homogeneous inviscid incompressible fluids}.
Comm. Partial Differential Equations, {\bf 37} (2012), n. 9, 1553-1595.

\bibitem{F-Feir_2021} F. Fanelli, E. Feireisl:
{\it Some remarks on steady solutions to the Euler system in $\R^d$}.
Appl. Math. Lett., {\bf 116} (2021), Paper n. 107031.

\bibitem{F-GB-S} F. Fanelli, R. Granero-Belinch\'on, S. Scrobogna:
{\it Well-posedness and singularity formation for the Kolmogorov two-equation model of turbulence in 1-D}.
J. Math. Pures Appl. (9), {\bf 187} (2024), 58-137.

\bibitem{HK} T. Hmidi,  S. Keraani:
{\it Incompressible viscous flows in borderline Besov spaces}.
Arch. Ration. Mech. Anal., {\bf 189} (2008), n. 2,  283-300.

\bibitem{Maj-Bert} A. J. Majda, A. L. Bertozzi:
{\it ``Vorticity and incompressible flow''}.
Cambridge Texts Appl. Math., 27 Cambridge University Press, Cambridge, 2002.

\bibitem{M-2008} G. M\'etivier:
{\it ``Para-differential calculus and applications to the Cauchy problem for nonlinear systems''}.
CRM Series, 5. Edizioni della Normale, Pisa, 2008.

\bibitem{Mey-Seis} D. Meyer, C. Seis:
{\it Propagation of regularity for transport equations. A Littlewood-Paley approach}.
Indiana Univ. Math. J., {\bf 73} (2024), n. 2, 445-472.

\bibitem{Paicu-Z} M. Paicu, P. Zhang:
{\it Striated regularity of 2-D inhomogeneous incompressible Navier-Stokes system with variable viscosity}.
Comm. Math. Phys., {\bf 376} (2020), n. 1, 385-439.

\bibitem{Vis} M. Vishik:
{\it Hydrodynamics in Besov spaces}.
Arch. Ration. Mech. Anal., {\bf 145} (1998), n. 3, 197-214.

\bibitem{Wol}  W. Wolibner:
{\it Un th\'eor\`eme d'existence du mouvement plan d'un fluide parfait, homog\`ene, incompressible, pendant un temps
infi\-ni\-ment long}.
{Math. Z.}, {\bf 37} (1933), n. 1, 698-726.

\end{thebibliography}

}
\end{document}